\documentclass[a4paper,10pt,twoside]{article}
\usepackage[ngerman,english]{babel}
\usepackage[T1]{fontenc}
\usepackage{amsmath,amstext,amsthm,amsfonts,amssymb}
\usepackage{bbm}
\usepackage{dsfont}
\usepackage{nicefrac}
\usepackage{colortbl,color,graphics,graphicx}
\usepackage{fancyhdr}
\usepackage{a4wide}
\usepackage{eucal}
\usepackage{times}
\usepackage{natbib}
\usepackage[clearempty]{titlesec}
\usepackage[pdflinkmargin=5pt,pdfstartview={FitBH -32768},plainpages=false]{hyperref}
\usepackage{graphics,graphicx,color,epsfig}
\usepackage[all]{xy}
\usepackage{comment}
\renewcommand{\:}{\mathrel{\mathop{:}}}

\renewcommand{\O}{\mathcal{O}}

\renewcommand{\P}{\mathbb{P}}
\newcommand{\E}{\mathbb{E}}
\newcommand{\N}{\mathds{N}}
\newcommand{\R}{\mathds{R}}

\newcommand{\C}{\mathds{C}}
\newcommand{\KLEINO}{{\scriptstyle{\mathcal{O}}}}

\renewcommand{\1}{\mathbbm{1}}

\renewcommand{\H}{\mathds{H}}

\newcommand{\ii}{\ensuremath{\mbox{\usefont{U}{eur}{m}{n} i}}} 
\DeclareSymbolFont{largesymbols}{OMX}{yhex}{m}{n}
\DeclareMathAccent{\verywidehat}{\mathord}{largesymbols}{'144}

\newcommand{\var}{\mathbb{V}\hspace*{-0.05cm}\textnormal{a\hspace*{0.02cm}r}}

\newtheorem{defi}{Definition}
\newtheorem{remark}{Remark}
\newtheorem{prop}{Proposition}
\newtheorem{lem}{Lemma}
\newtheorem{theo}{Theorem}

\newtheorem{assump}{Assumption}

\makeatletter
\def\Links{\tagsleft@true}\def\Rechts{\tagsleft@false}
\makeatother
\def\lsim{\mathrel{\rlap{\lower4pt\hbox{\hskip1pt$\sim$}}
    \raise1pt\hbox{$<$}}}                
\def\gsim{\mathrel{\rlap{\lower4pt\hbox{\hskip1pt$\sim$}}
    \raise1pt\hbox{$>$}}}
\providecommand{\eps}{\varepsilon}
\providecommand{\floor}[1]{\lfloor #1 \rfloor}

\begin{document}
\selectlanguage{english}
\pagestyle{fancy}
\title{Spectral covolatility estimation from noisy observations using local weights}
\fancyhf{}
\lhead[\thepage ~~~~\textsl{M.\,Bibinger\,\&\,M.\,Rei\ss}]{}
\rhead[]{\textsl{Spectral covolatility estimation}~~~~ \thepage}
\huge \noindent
\textbf{Spectral estimation of covolatility\\ from noisy observations using local weights}\\[.5cm]
\selectlanguage{ngerman}
\Large  Markus Bibinger \footnote{Financial support from the Deutsche Forschungsgemeinschaft via SFB 649 `\"Okonomisches Risiko', Humboldt-Universität zu Berlin, is gratefully acknowledged.}
\,\&\,Markus Rei\ss\footnotemark[1]\\[.5cm]
\selectlanguage{english}
\large \noindent
\textsl{Institute of Mathematics,
Humboldt-Universität zu Berlin}\\[.4cm]
 \normalsize\textbf{ABSTRACT.
We propose localized spectral estimators for the quadratic covariation and the spot covolatility of diffusion processes which are observed discretely with additive observation noise. The appropriate estimation for time-varying volatilities is based on an asymptotic equivalence of the underlying statistical model to a white noise model with correlation and volatility processes being constant over small time intervals. The asymptotic equivalence of the continuous-time and the discrete-time experiments is proved by a construction with linear interpolation in one direction and local means for the other. The new estimator outperforms earlier nonparametric methods in the literature for the considered model. We investigate its finite sample size characteristics in simulations and draw a comparison between various proposed methods.} \\[.4cm]
\noindent
\normalsize
\section{Introduction\label{sec:1}}
The estimation of the quadratic (co-)variation of semimartingales is of large interest in statistics and financial econometrics. Especially, statistical models taking market microstructure frictions into account have attracted a lot of attention in recent years. Inspired by empirical studies of the characteristics of high-frequency financial data, a prominent approach is to describe asset prices as a superposition of a discretely sampled semimartingale with an independent additive noise component.\\
The finding that $n$ observations of a Brownian motion with noise on a discrete time grid possesses the LAN-property in Le Cam's sense with the rate $n^{-\nicefrac{1}{4}}$ by \cite{gloter}, instead of the usual $n^{-\nicefrac{1}{2}}$ rate in the absence of noise, has provided the optimal rate and a parametric efficiency bound for the asymptotic variance as a benchmark for this estimation problem. Interestingly, the nuisance quantity, namely the noise level, can be estimated with the usual faster rate in this model in contrast to the parameter of interest. This is caused by observation errors with non-decreasing variances perturbing diffusion increments of order $n^{-\nicefrac{1}{2}}$. These features carry over to the estimation problem of covariation in a multidimensional setting as has been shown in \cite{bibinger}.\\
The key role of quantifying integrated (co-)volatilities in portfolio optimization and risk management has stimulated an increasing interest in estimation methods for these models starting with \cite{howoften} and \cite{zhangmykland}. Subsequently, at least three nonparametric approaches for integrated volatility estimation have been suggested, the multi-scale realized volatility by \cite{zhang}, a pre-average strategy by \cite{JLMPV} and the realized kernels from \cite{bn2}. All estimators are based on quadratic forms of the observations and depend on a globally chosen tuning parameter. For that reason, when ignoring the treatment of end-effects, all three share a similar asymptotic behavior. They attain the optimal rate, but cannot be asymptotically efficient for time-varying volatility functions.
Still, several robustness results to more realistic models incorporating non-i.\,i.\,d.\,noise and stochastic volatilities with leverage have been established and make these approaches quite attractive.

An alternative approach for the estimation of the quadratic variation arising in \cite{howoften} from the parametric point of view is based on the MLE for this model. It turned out in \cite{xiu} that the MLE for integrated volatility can cope with a nonparametric volatility specification. This quasi-maximum likelihood estimator (QMLE) also attains the optimal rate. Asymptotic efficiency, however, is achieved only in the parametric setup with constant volatility. In \cite{reiss} an asymptotically efficient estimator based on spectral theory and localized parametric estimation for asymptotically shrinking blocks has been constructed. The idea stems from an asymptotic equivalence result in the spirit of \cite{gramanussbaum} pertaining the underlying nonparametric setting and a piecewise constant local parametric approximation. In \cite{corsi} a related estimation strategy for integrated volatility using a discrete sine transform approach with the eigenvectors of an equidistantly observed parametric model is considered and tested in an application study. \\
Ongoing progress in this research area has recently led to estimation approaches for the integrated covolatility in multidimensional models. The above-mentioned methods carry over to a multidimensional setting. Rate-optimal estimators, which also cope with asynchronous observations, have been established by \cite{kinnepoldi}, \cite{sahalia} and \cite{bibinger}, while \cite{bn1} focusses on positive-definite (co)volatility matrix estimators.\\
The motivation and contribution of the article at hand is twofold. First, we step forward towards a deeper understanding of the statistical properties of covariation estimation from noisy discretely observed diffusions. In particular, we prove that observing two correlated diffusion processes with noise at synchronous times is asymptotically equivalent (in the sense of Le Cam's equivalence of statistical experiments)  to observations in a related continuous time white noise model. The procedure is completely explicit and thus allows to transfer estimators and tests from one model to the other with the same asymptotic properties.  In particular, for bounded loss functions asymptotic efficiency results are the same in both model sequences. The white noise model itself is asymptotically equivalent to a piecewise constructed parametric model. That result is an extension of the one-dimensional findings in \cite{reiss} and gives rise to our local spectral approach. The second contribution is provided by our nonparametric spectral estimators of covolatility (SPECV) for both, the integrated covolatility (i.e. covariation) and the spot (i.e. instantaneous) covolatility. The estimators are based on certain empirical bivariate Fourier coefficients on each block in time which in the piecewise parametric white noise model are just independent Gaussian vectors in $\R^2$ with volatilities and covolatilities appearing in the covariance structures. This very simple structure allows a straight-forward analysis and  often reduces the estimation variance significantly compared to the previously suggested methods. This is corroborated by simulation results which show good finite-sample properties.\\
The article is arranged in three upcoming sections and an appendix comprising the technical proofs.
Section \ref{sec:2} is devoted to the underlying statistical experiments and the asymptotic equivalence results.
In Section \ref{sec:3}, we develop the SPECV, spectral estimator of covolatility, and investigate its mathematical properties. A discussion and simulation study is provided in Section \ref{sec:4}, where the SPECV of integrated covolatility is compared to concurrent nonparametric approaches. Owing to its local spectral construction principle, the new approach outperforms earlier methods if the correlation or volatility processes vary in time.

\section{Asymptotic equivalence of the discrete regression-type and the continuous white noise experiment\label{sec:2}}
Consider the statistical experiment in which a two-dimensional discrete time process $\tilde{\mathbf{Z}}$ defined by
\Links
\begin{align}\label{E0}\tag{$\mathcal{E}_0$}\tilde{\mathbf{Z}}_{t_i^n}=\mathbf{Z}_{t_i^n}+\mathbf{\eps}_i,\,0\le i\le n,~\mbox{with}~\mathbf{Z}_t=\mathbf{Z}_0+\int_0^t\Sigma_s^{\nicefrac{1}{2}}\,d\mathbf{B}_s,\,t\in[0,1]\end{align}\Rechts
is observed, where $\mathbf{B}$ is a two-dimensional standard Brownian motion and
\begin{align*}\Sigma_t=\left(\begin{array}{cc} (\sigma_t^X)^2 & \rho_t\sigma_t^X\sigma_t^Y\\[.135cm]  \rho_t\sigma_t^X\sigma_t^Y &(\sigma_t^Y)^2\end{array}\right)\end{align*}
the (spot) volatility matrix. The signal part of $\tilde{\mathbf{Z}}=(\tilde X,\tilde Y)^{\top}$, denoted $\mathbf{Z}=\left(X,Y\right)^{\top}$, and called efficient price process in finance, is supposed to be independent of the observation noise $\mathbf{\eps}=(\eps^X,\eps^Y)^{\top}$.
The observation errors $(\mathbf{\eps}_i)$ are i.\,i.\,d.\, centred normal with covariance matrix
\begin{align*}\mathbf{H}=\left(\begin{array}{cc} \eta_X^2 & \eta_{XY}\\  \eta_{XY} &\eta_Y^2\end{array}\right)\,.\end{align*}
We consider time-varying volatility matrices $\Sigma$ belonging to a Hölder ball of order $\alpha\in(0,1]$ and radius $R>0$, i.e. $\Sigma\in C^\alpha(R)$ with
\begin{align*}&C^{\alpha}(R)=\{f\in C^{\alpha}([0,1],\mathds{R}^{2\times 2})|\|f\|_{C^{\alpha}}\le R\}\text{ where }\|f\|_{C^{\alpha}}\:=\|f\|_{\infty}+\sup_{x\ne y}{\frac{\|f(x)-f(y)\|}{|x-y|^{\alpha}}}\,.\end{align*}
We denote the 
spectral norm in $\R^{2\times 2}$ always by $\|\cdot\|$ and define $\|f\|_{\infty}\:=\sup_{t\in[0,1]}{\|f(t)\|}$.\\
In \eqref{E0} we allow for a non-equidistant synchronous observation scheme $(t_i^n)_{0\le i\le n}$, but we assume 
that the sampling can be transferred to an equidistant scheme by a quantile transformation independent of $n$.

\begin{assump}\label{sampling}Suppose that there exists a 
differentiable distribution function $F:[0,1]\rightarrow [0,1]$ with $F(0)=0$, $F(1)=1$ and $F'>0$, such that the observation times of $\tilde{\mathbf{Z}}=(\tilde X,\tilde Y)^{\top}$ in \eqref{E0} are generated by  $t_i^n=F^{-1}(i/n),\,i=0,\ldots,n$.
\end{assump}

Note that we only consider deterministic designs of observation times. Under a random sampling scheme, with c.\,d.\,f.\,$F$, independent of $\tilde{\mathbf{Z}}$, the estimators are expected to have similar properties, but the mathematical analysis is much harder.\\
We use a similar notation for the white noise experiment
\Links
\begin{align}\label{E1}\tag{$\mathcal{E}_1$}d\tilde{\mathbf{Z}}_t=\mathbf{Z}_t\,dt+n^{-\nicefrac{1}{2}}\mathbf{H}^{\nicefrac{1}{2}} \,d\mathbf{W}_t\,,t\in[0,1]\,,\end{align}\Rechts
with the covariance matrix $\mathbf{H}$ of $\eps$, $\mathbf{Z}_{t}=\mathbf{Z}_0+\int_0^t\Sigma_{s}^{\nicefrac{1}{2}}\,d\mathbf{B}_s$ and a standard two-dimensional Brownian motion $\mathbf{W}$ independent of $\mathbf{B}$.\\
In the following, we shall prove the results for an equidistant setting $t_i^n=i/n,i=0,\ldots,n$. This is founded on the connection between a sampling scheme based on a quantile transformation of the equidistant grid and an equidistantly observed process with transformed volatility matrix by the identity in law
\begin{align*}
\mathbf{Z}^F_u:=\mathbf{Z}_{F^{-1}(u)}\stackrel{d}{=}\int_0^u(\Sigma^F_s)^{\nicefrac12}\,d\mathbf{B}_s\text{ with } \Sigma^F_s=\Sigma_{F^{-1}(s)}(F^{-1})^{\prime}(s),
\end{align*}
which follows directly from the identity for covariance functions of these Gaussian processes via It\^o isometry.
Hence, upcoming results can be generalized for all $F$ satisfying Assumption \ref{sampling}, replacing everywhere $\mathbf{Z}$ by $\mathbf{Z}^F$, $i/n$ by $t_i^n$ and $\Sigma$ by $\Sigma^F$.
Yet, the ease in dealing with transformations in the white noise model even gives another useful representation for non-equidistant design. Experiment $(\mathcal{E}_1)$ in terms of observing $\mathbf{Z}^F$ in noise is equivalent to observing
\[ d\tilde{\mathbf{Z}}_{F(t)}=\mathbf{Z}_tF'(t)\,dt+n^{-\nicefrac{1}{2}}\mathbf{H}^{\nicefrac{1}{2}}F'(t)^{1/2} \,d\mathbf{W}_t\,,t\in[0,1],
\]
see below for the exact notion of Le Cam equivalence which can be easily verified here by the identity of likelihood processes.
Dividing by $F'(t)$ yields further equivalence with observing
\begin{equation}\label{EqWNMF} d\overline{\mathbf{Z}}_{t}=\mathbf{Z}_t\,dt+(nF'(t))^{-\nicefrac{1}{2}}\mathbf{H}^{\nicefrac{1}{2}}\,d\mathbf{W}_t\,,t\in[0,1].
\end{equation}
Interpreting $F^{\prime}$ as a design density, $(nF^{\prime}(t))^{-1/2}$ can be understood as the local noise level induced by the local sample size $nF^{\prime}$.
As we shall establish next, experiments $(\mathcal{E}_0)$ and $(\mathcal{E}_1)$ will be asymptotically equivalent for $n\to\infty$.
 
\begin{defi}
Let $\mathcal{E}_0(n,\alpha,R,\underline{\Sigma})$ with $n\in\mathds{N},\alpha\in(0,1],R,\underline{\Sigma}\ge 0$, be the statistical experiment generated by observations from \eqref{E0} with $t_i^n=i/n$. The unknown parameter $\Sigma$ in \eqref{E0} belongs to the class $C^{\alpha}(R)$ and satisfies $\Sigma_t\ge \underline{\Sigma}E_2$ for all $t\in[0,1]$ with the identity matrix $E_2\in\mathds{R}^{2\times 2}$, i.e.  the smallest eigenvalues of $\Sigma_t$ are bounded from below by $\underline{\Sigma}$.

Analogously, let $\mathcal{E}_1(n,\alpha,R,\underline{\Sigma})$ be the statistical experiment generated by observing \eqref{E1} with the same parameter class for $\Sigma$.
\end{defi}
Throughout this article we consider deterministic (squared) volatility functions $\Sigma_t,t\in[0,1]$, with the smallest eigenvalue $\underline{\Sigma}$, which is assumed to be bounded uniformly from below by a positive constant for our second equivalence result to a locally parametric experiment. We impose sufficient smoothness for $\Sigma$, in terms of Hölder continuity, such that equivalence of $\mathcal{E}_0$ to $\mathcal{E}_1$ and $\mathcal{E}_2$, respectively, holds. The first equivalence is valid for any Hölder exponent $\alpha>0$ and the block-wise constant approximation for $\alpha>1/2$.\\
For the following results, we will briefly recall the notion of asymptotic equivalence, Le Cam deficiency and Le Cam distance. We refer interested readers to \cite{lecamyang} for more information on the underlying theory.
For statistical experiments $\mathcal{E}_0=\left(\mathcal{X}_0,\mathcal{F}_0,\{\P^0_{\theta}\,|\theta\in\Theta\}\right)$ and $\mathcal{E}_1=\left(\mathcal{X}_1,\mathcal{F}_1,\{\P^1_{\theta}\,|\theta\in\Theta\}\right)$ with the same parameter set $\Theta$ defined on (possibly) different Polish spaces, their Le Cam deficiency is 
defined by
\begin{align*}\delta\left(\mathcal{E}_0,\mathcal{E}_1\right)=\inf_K{\sup_{\theta\in\Theta}\|K\P_{\theta}^0-\P_{\theta}^1\|_{\text{TV}}}\,,\end{align*}
where the infimum is taken over all Markov kernels (or randomisations) $K$ from $(\mathcal{X}_0,\mathcal{F}_0)$ to $(\mathcal{X}_1,\mathcal{F}_1)$.
The Le Cam distance is defined by
\begin{align*}\Delta\left(\mathcal{E}_0,\mathcal{E}_1\right)=\max{\left(\delta\left(\mathcal{E}_1,\mathcal{E}_0\right),\delta\left(\mathcal{E}_0,\mathcal{E}_1\right)\right)}\,.\end{align*}
If
\begin{align*}\lim_{n\rightarrow\infty}\Delta \left(\mathcal{E}_0^n,\mathcal{E}_1^n\right)=0\end{align*}
holds for sequences of experiments $\left(\mathcal{E}_0^n\right)_n$ and $\left(\mathcal{E}_1^n\right)_n$, then these sequences are called asymptotically equivalent.
The construction of the Markov kernel $K$ will be explicit in all the proofs given in this article in terms of data transformations and randomisations, which allows a constructive way to transfer methods from one model to the other and vice versa.

\begin{theo}\label{theo1}
The statistical experiments $\mathcal{E}_0(n,\alpha,R,\underline{\Sigma})$ and $\mathcal{E}_1(n,\alpha,R,\underline{\Sigma})$ are for any $\alpha>0,\underline{\Sigma}\ge 0,R>0$ and $n\rightarrow\infty$ asymptotically equivalent. More precisely, the Le Cam distance is of order
\begin{align}\Delta\left(\mathcal{E}_0,\mathcal{E}_1\right)=\mathcal{O}\left(Rn^{-(\alpha\wedge \nicefrac12)}\underline{\mathbf{H}}^{-1}\right)\,,\end{align}
where $\underline{\mathbf{H}}>0$ denotes the smallest eigenvalue of $\mathbf{H}$.
\end{theo}

We explicitly state how asymptotic terms hinge on $\underline{\mathbf{H}}$, since this is of interest when considering noise levels decreasing with $n$.
A concise proof of this theorem is given in the appendix. The strategy of proof follows the same principle as for the one-dimensional setting in \cite{reiss}. For the proof that \eqref{E0} is at least as informative as \eqref{E1}, we construct a continuous time observation by linear interpolation. The interpolated process $\hat{\mathbf{Z}}$ is a centred Gaussian process on $[0,1]$. The associated covariance operator $\hat C$ on $L^2\left([0,1],\mathds{R}^2\right)$ is such that the difference $(\bar C-\hat C)$, where $\bar C$ is the covariance operator in a white noise model comprising the interpolated signal term, is positive (semi-)definite. For this reason observations from such a white noise model can be generated by adding an independent Gaussian noise component to $\hat{\mathbf{Z}}$.
Now a process $\bar {\mathbf{Z}}$ from this white noise model and $\tilde {\mathbf{Z}}$ in \eqref{E1} can be defined on the same probability space and it suffices to show that the total variation distance of the laws converges uniformly over $\Sigma$ to zero.
This is accomplished by bounding the squared Hellinger distance. For the proof of the intuitive converse, that \eqref{E1} is at least as informative as \eqref{E0}, we take means symmetrically around the points $(i/n),\,1\le i\le (n-1)$, from \eqref{E1} and verify that the Hellinger distance between the processes generated in this manner and $\tilde {\mathbf{Z}}$ from \eqref{E0} tends to zero.\\
The important setting in which the volatility processes follow again continuous semimartingales is covered by Theorem \ref{theo1} for the case that $\mathbf{Z}$ remains conditionally Gaussian.

\begin{defi}Write $\lfloor t\rfloor_h=\lfloor t/h\rfloor h$ for $h>0$, assume $h^{-1},nh\in\mathds{N}$ and let $\mathbf{Z}_t^h=\mathbf{Z}_0+\int_0^t\Sigma^{\nicefrac{1}{2}}_{\lfloor s\rfloor_h}\,d\mathbf{B}_s$ with a two-dimensional standard Brownian motion $\mathbf{B}$. Let $\Sigma_t$ belong to $C^{\alpha}(R)$ and satisfy $\Sigma_t\ge \underline{\Sigma}E_2$. Define the process
\Links
\begin{align}\label{E2}\tag{$\mathcal{E}_2$}d\tilde{\mathbf{Z}}_t=\mathbf{Z}_t^h\,dt+n^{-\nicefrac{1}{2}}\mathbf{H}^{\nicefrac{1}{2}}\,d\mathbf{W}_t,\,t\in[0,1],\end{align}\Rechts
where $\mathbf{W}$ is a standard Brownian motion independent of $\mathbf{B}$. The statistical model generated by the observations from \eqref{E2} is denoted by $\mathcal{E}_2(n,h,\alpha,R,\underline{\Sigma})$.
\end{defi}

In experiment \eqref{E2} we thus observe a process with a volatility matrix which is constant on each block $[kh,(k+1)h)$, $k=0,1,\ldots,h^{-1}-1$. It is intuitive that for small block sizes $h$ and sufficient H\"older regularity $\alpha$ this piecewise constant approximation is sufficiently close to render the approximation error statistically negligible. This is made precise in the following theorem. Note that a piecewise constant approximation is common in this research field, but so far the only general statistical approximation result has been established by \cite{lamanga12} where, however, only contiguity and block lengths of order $n^{-1}$ are considered, while allowing for stochastic volatility. Here we need to establish convergence in total variation norm for blocks, with order $nh$ observations within each block, over which we smooth the noise perturbation. \\

\begin{theo}\label{theo2}
Assume $h^{\alpha}=\KLEINO\left(n^{-\nicefrac{1}{4}}\right)$ for $\alpha\in(1/2,1]$ and $\underline{\Sigma}>0$. Then the statistical experiments $\mathcal{E}_1(n,\alpha,R,\underline{\Sigma})$ and $\mathcal{E}_2(n,h,\alpha,R,\underline{\Sigma})$ are asymptotically equivalent:
\begin{align}\Delta\left(\mathcal{E}_1,\mathcal{E}_2\right)
=\mathcal{O}\left(Rh^{\alpha}\underline{\Sigma}^{-3/4}\underline{\mathbf H}^{-1/4}n^{\nicefrac14}\right)\,.\end{align}
\end{theo}

The asymptotic equivalence results lead to a new approach for the covariation estimation problem.
Following the idea for the one-dimensional case from \cite{reiss}, we consider an orthonormal system $(\varphi_{jk})$ in $L^2([0,1])$ of cosine functions with support on the blocks $[kh,(k+1)h]$ and frequencies of order $j\ge 1$. Their antiderivatives $(\Phi_{jk})$ are sine functions on the same support and will also play a crucial role. We set
\begin{subequations}
\begin{align}\label{varphi}\varphi_{jk}(t)=\sqrt{\frac{2}{h}}\cos{\left(j\pi h^{-1}\left(t-kh\right)\right)}\1_{[kh,(k+1)h]}(t),\quad j\ge 1\,,k=0,\ldots,h^{-1}-1\,,\end{align}
\begin{align}\label{Phi}\Phi_{jk}(t)=\hspace*{-.1cm}\left(\hspace*{-.1cm}\sqrt{2h}\,n\sin{\hspace*{-.075cm}\left(\frac{j\pi}{2nh}\right)}\right)^{-1}\hspace*{-.1cm}\sin{\left(j\pi h^{-1}\hspace*{-.075cm}\left(t-kh\right)\right)}\1_{[kh,(k+1)h]}(t),\, j\ge 1\,,k=0,\ldots,h^{-1}-1\,.\end{align}
\end{subequations}
Differently from \cite{reiss}, we renormalize the antiderivatives \eqref{Phi} exactly for the discrete analysis.
The functions \eqref{varphi} and \eqref{Phi}, evaluated on the grid given by the observation times, provide spectral weights for local blockwise averages. By virtue of the transformation for general observation schemes discussed above, we may for ease of exposition consider the equidistant grid:
\begin{subequations}
\begin{align}\label{bl1}\tilde x_{jk}=\sum_{l=1}^n\left(\tilde X_{\frac{l}{n}}-\tilde X_{\frac{l-1}{n}}\right)\Phi_{jk}\left(\frac{l}{n}\right),\end{align}
\begin{align}\label{bl2}\tilde y_{jk}=\sum_{l=1}^n\left(\tilde Y_{\frac{l}{n}}-\tilde Y_{\frac{l-1}{n}}\right)\Phi_{jk}\left(\frac{l}{n}\right).\end{align}
\end{subequations}
Since $\Phi_{j(h^{-1}-1)}(1)=0$, the last addend is zero for all blocks $k$. We stress that by the indicator functions in \eqref{varphi} and \eqref{Phi} and since $\Phi_{jk}(kh)=\Phi_{jk}((k+1)h)=0$, the sums in \eqref{bl1} and \eqref{bl2} only extend over $l=k\cdot nh+1,\,\ldots,\,(k+1)\cdot nh-1$. Therefore, families $(\tilde x_{jk},\tilde y_{jk})_j$ are uncorrelated and thus by Gaussianity independent for different blocks $k$. Besides the independence between blocks, we additionally benefit from the orthogonality of each family of functions associated with a specific frequency. The orthogonality relations $\int \varphi_{jk}\varphi_{ik}=0$ and $\int \Phi_{jk}\Phi_{ik}=0$ $\forall i\ne j$ in $L^2([0,1])$ will remain valid for the discretized versions and the corresponding sums when $i,j\in\{1,\ldots,nh\}$.
For the purpose of explicitly analyzing the discrete terms, we introduce the notion of empirical scalar products:
\begin{subequations}
\begin{align}\label{sc1}\langle f,g\rangle_n\:=\frac{1}{n}\sum_{l=1}^n f\left(\frac{l}{n}\right)g\left(\frac{l}{n}\right)~\mbox{and}~\|f\|_n^2\:=\frac{1}{n}\sum_{l=1}^nf^2\left(\frac{l}{n}\right)=\langle f,f\rangle_n\,,\end{align}
\begin{align}\label{sc2}\left[ f,g\right]_n\:=\frac{1}{n}\sum_{l=1}^n f\left(\frac{l-\frac{1}{2}}{n}\right)g\left(\frac{l-\frac{1}{2}}{n}\right), \;\mbox{for}\,f,g:[0,1]\rightarrow\mathds{R}\,.\end{align}
By abuse of notation for a vector $Z=(Z_1,\ldots,Z_n)$ and $f:[0,1]\rightarrow\mathds{R}$, we will also write
\begin{align}\langle Z,f\rangle_n\:=\frac{1}{n}\sum_{l=1}^nZ_l\,f\left(\frac{l}{n}\right)~\mbox{and}~\left[Z,f\right]_n\:=\frac{1}{n}\sum_{l=1}^nZ_l\,f\left(\frac{l-\frac{1}{2}}{n}\right)\,.\end{align}
For two such vectors $Z$ and $\tilde Z$ it is convenient to introduce the notation
\begin{align}\langle Z,\tilde Z\rangle_{nh;k}\:=\frac{1}{n}\sum_{l=1}^nZ_l \tilde Z_l \1_{[kh,(k+1)h]}\left(\frac{l}{n}\right)=\frac{1}{n}\sum_{i=0}^{nh}Z_{knh+i}\tilde Z_{knh+i}\,.\end{align}
\end{subequations} \noindent
The following identity from discrete Fourier analysis is a main ingredient in the construction of the estimator and for its error analysis below.

\begin{prop}\label{propsbp}
For the blockwise weighted sums $\tilde x_{jk},\tilde y_{jk}\,,j\in\{1,\ldots,nh\},k\in\{0,\ldots,h^{-1}-1\}$, the following summation by parts formula holds true:
\begin{align}\label{sbp}\tilde y_{jk}&=\sum_{l=1}^{n}\Delta \tilde Y_{l}\,\Phi_{jk}\left(\frac{l}{n}\right)=-\sum_{l=0}^{n-1}\tilde Y_\frac{l}{n}\,\varphi_{jk}\left(\frac{l+\frac{1}{2}}{n}\right)\frac{1}{n}\\
&\notag=\langle n \Delta Y,\Phi_{jk}\rangle_n-\left[\eps^Y,\varphi_{jk}\right]_n \,,\end{align}
and for $\tilde x_{jk}$ analogously, where $\Delta$ denotes the backward difference operator $\Delta \tilde Y_l\:=\tilde Y_{\frac{l}{n}}-\tilde Y_{\frac{l-1}{n}}$ and $\Delta \tilde Y=\left(\Delta \tilde Y_1,\ldots,\Delta \tilde Y_n\right)$. Moreover, we have the following orthogonality identities:
\begin{subequations}
\begin{align}\label{o1}\left[ \varphi_{jk},\varphi_{rk}\right]_n=\delta_{jr}\,,~j,r\in\{1,\ldots,nh\}\,,k=0,\ldots,h^{-1}-1\,,\end{align}
\begin{align}\label{o2}\langle \Phi_{jk},\Phi_{rk}\rangle_n=\|\Phi_{jk}\|_n^2\,\delta_{jr}\,,~j,r\in\{1,\ldots,nh\}\,,k=0,\ldots,h^{-1}-1\,,\end{align}
where $\delta_{jr}=\1_{\{j=r\}}$ is Kronecker's delta.
\end{subequations}
The empirical norm
\begin{align}\label{empnorm}\|\Phi_{jk}\|_n^2=\left(4n^2\sin^2{(j\pi/(2nh))}\right)^{-1},\,k\in\{0,\ldots,h^{-1}-1\},\end{align}
does not depend on the block $k$ and appears in our estimator in the next section.
\end{prop}
The two representations of the blockwise sums in \eqref{sbp} are very useful when disentangling the estimation error emerging from the two independent error sources: discretization and observation noise. In particular, we use the left-hand side which involves the increments of the processes only when considering the signal parts $X$ and $Y$. For the analysis of cross terms and the pure noise parts the right-hand side of \eqref{sbp} permits a significant simplification.
In the next section, we use these ideas and the insight into the structure of the estimation problem to construct a new estimation approach for the quadratic covariation and the spot covolatility of diffusion processes based on the original model \eqref{E0}. The final estimator for the quadratic covariation appears as a linear combination of the products of the local spectral averages $\tilde x_{jk}\tilde{y}_{jk}$ over all $j$ and $k$ combined with a bias correction. We will benefit from the asymptotic equivalence results for the mathematical analysis of our estimator by the following conclusion that we can straiten the analysis to the statistical experiment
\Links
\begin{align}\label{E3}\tag{$\mathcal{E}_3$}\tilde{\mathbf{Z}}^h_{t_i^n}=\mathbf{Z}^h_{t_i^n}+\mathbf{\eps}_i,\,0\le i\le n~\mbox{with}~\mathbf{Z}^h_t=\mathbf{Z}_0+\int_0^t\Sigma_{\lfloor s\rfloor_h}^{\nicefrac{1}{2}}\,d\mathbf{B}_s,\,t\in[0,1]\,,\end{align}\Rechts
where we have noisy discrete observations with the volatility matrix being constant on blocks.

\begin{prop}\label{propE3} For $nh\in\N$, $\alpha,R>0$ and $\underline{\Sigma}\ge 0$ the statistical experiments $\mathcal{E}_2(n,h,\alpha,R,\underline{\Sigma})$ and $\mathcal{E}_3(n,h,\alpha,R,\underline{\Sigma})$ with $t_i^n=i/n$ are asymptotically equivalent:
\begin{align}\Delta\left(\mathcal{E}_2,\mathcal{E}_3\right)
=\mathcal{O}\left(R\underline{\mathbf{H}}^{-1}n^{-\nicefrac12}\right)\,.\end{align}
\end{prop}

Consequently, observing $\tilde {\mathbf{Z}}$ in \eqref{E0} is asymptotically equivalent to observations of $\tilde {\mathbf{Z}}^h$ from \eqref{E3}.
Note that for constant $\Sigma_{kh}$ on each block, the $\Phi_{jk}$ have the same structure as the eigenvectors of the covariance matrix associated with the vector of the $(nh-1)$ observed increments on the block. In particular, inserting for $h=1, k=0$, the discrete grid $t=l/(n+1),l=0,\ldots,n$, in \eqref{Phi} gives, apart from a normalizing factor, the basis used by \cite{corsi}, i.\,e.\,the eigenfunctions of the covariance operator in the parametric model. The local weighted sums \eqref{bl1} and \eqref{bl2} on each block hence constitute the corresponding Karhunen-Lo\`{ev}ve expansion. We refer to \cite{bibinger} and for the one-dimensional case to \cite{gloter} and \cite{corsi} for the explicit computation of the eigenvalues.

\section{Local spectral estimation of covolatility\label{sec:3}}

In the sequel, we always assume $h^{\alpha}=\KLEINO\left(n^{-\nicefrac{1}{4}}\right)$, Assumption \ref{sampling} on the sampling scheme and that the volatility matrix belongs to $C^{\alpha}(R)$ for some $\alpha>1/2$ and $R>0$ and is bounded from below by $\underline{\Sigma}E_2$ with $\underline{\Sigma}>0$. By virtue of Proposition \ref{propE3}, we can then work within the simpler model \eqref{E3}. We present all results for the equidistant design $t_i^n=i/n$, noting again that the general case follows by substituting $\tilde{\mathbf{Z}}$ by $\tilde{\mathbf{Z}}^F$, $\Sigma$ by $\Sigma^F$ etc. Interestingly, integrated volatility is even invariant under this transformation:
\[ \int_0^1 \Sigma^F_u\,du=\int_0^1 \Sigma_{F^{-1}(u)}(F^{-1})'(u)\,du=\int_0^1 \Sigma_t\,dt.\]
For estimation purposes this means that we can just neglect the design in the implementation and work in ``tick time''. The invariance property, however, does not hold for powers of $\Sigma$ or for polynomials in $\sigma^X,\sigma^Y$ of degree different from two such that the asymptotic variance will depend on the design function $F$.\\
On each of the independent blocks, we have observations \eqref{bl1} and \eqref{bl2} with 
\begin{align}\label{nov}\left( \tilde x_{jk},\tilde y_{jk}\right)\sim\mathbf{N}\left(\mathbf{0}\,,\left(\begin{array}{cc} \eta_X^2/n+\|\Phi_{jk}\|_n^2(\sigma_{kh}^X)^2 & \eta_{XY}/n+\|\Phi_{jk}\|_n^2\rho_{kh}\sigma_{kh}^{X}\sigma_{kh}^Y\\  \eta_{XY}/n+\|\Phi_{jk}\|_n^2\rho_{kh}\sigma_{kh}^{X}\sigma_{kh}^Y & \eta_Y^2/n+\|\Phi_{jk}\|_n^2(\sigma_{kh}^Y)^2\end{array}\right)\right)\,,\end{align}
independently for all $j,k$, what can be proved using Proposition \ref{propsbp}. We will postpone a detailed computation of estimation errors to the Appendix \ref{sec:app2}.
For each $j,k$ fixed, the empirical covariance yields a natural estimator of the spot covolatility $\rho_{kh}\sigma_{kh}^{X}\sigma_{kh}^Y$ on each block provided we correct for the bias by subtracting $\eta_{XY}/n$. In particular, the independent statistics $(\tilde x_{jk},\tilde y_{jk})$ in \eqref{nov} form a Gaussian scale model from classical statistics and will suggest an optimal convex combination of empirircal covariances as an estimator for integrated covolatility.
\begin{remark}In the following we assume for the ease of exposition that $\eta_{XY}$ 
is known. Yet, we can estimate $\eta_{XY}$ from the observations with faster rate $\sqrt{n}$ by
\begin{align}\widehat{\eta_{XY}}=\frac{1}{2n}\sum_{l=1}^n\left(\tilde Y_{\frac{l}{n}}-\tilde Y_{\frac{l-1}{n}}\right)\left(\tilde X_{\frac{l}{n}}-\tilde X_{\frac{l-1}{n}}\right)\end{align}
or as well by $-n^{-1}\sum_l (\tilde Y_{{l/n}}-\tilde Y_{(l-1)/n})(\tilde X_{(l+1)/n}-\tilde X_{l/n})$. For the first estimator $\sqrt{n}$-consistency and a central limit theorem can be proved in the spirit of \cite{zhangmykland} for its one-dimensional counterpart $1/(2n)\sum_l (\tilde X_{l/n}-\tilde X_{(l-1)/n})^2$. The second estimator and its one-dimensional analogue $-n^{-1}\sum_l (\tilde X_{l/n}-\tilde X_{(l-1)/n})(\tilde X_{(l+1)/n}-\tilde X_{l/n})$ have a slightly bigger variance, but the benefit of no finite sample bias due to the quadratic (co-)variation of the signal part.
\end{remark}
By using just the lowest frequency $j=1$ in each block, we obtain a simple rate-optimal estimator of integrated covolatility when summing over all blocks $[kh,(k+1)h]$ multiplied by the block length $h$:
\begin{align}\widehat{IC}^{(SPECV,j=1)}=h\sum_{k=0}^{h^{-1}-1} \|\Phi_{1k}\|_n^{-2}\left(\tilde x_{1k}\tilde y_{1k}-\eta_{XY}/n\right)\,.
\end{align}
By independence between the blocks, its variance is of order $\O(h^{-3}(\eta_X^2/n+h^2)(\eta_Y^2/n+h^2))$. For fixed noise levels $\eta_X,\eta_Y,\eta_{XY}$, the rate-optimal choice $h\sim n^{-1/2}$ thus yields a variance of order $\O(n^{-1/2})$ (note that for $\alpha>1/2$ and $h\sim n^{-1/2}$ the condition $h^\alpha=\KLEINO(n^{-\nicefrac14})$ always holds).

It is possible to obtain a pointwise estimator of the spot covolatility $SCV_t:=\rho_t\sigma_t^{X}\sigma_t^Y$ by the average of the spectral estimators over a set ${\mathcal K}_t$ of $K$ adjacent blocks containing $t$:
\begin{equation}\label{PilotSCV} \widehat{SCV}_t=K^{-1}\sum_{k\in{\mathcal K}_t} \|\Phi_{1k}\|_n^{-2}\left(\tilde x_{1k}\tilde y_{1k}-\eta_{XY}/n\right)\,.
\end{equation}
Since the observation times in ${\mathcal K}_t$ have at most distance $Kh$ to $t$, the approximation error bound for the $\alpha$-H\"older continuous function $\Sigma$ yields a squared bias of
order $\O((Kh)^{2\alpha})$. The variance is $\O(K^{-1})$ for $h\gsim n^{-1/2}$, and we obtain for the rate-optimal choices $h\sim n^{-1/2}$, $K\sim n^{\alpha/(2\alpha+1)}$, a root mean squared error of order $\O(n^{-\alpha/(4\alpha+2)})$. Standard nonparametric techniques based on Gaussian measure concentration then even give the same rate times a log-factor in $n$ for uniform loss in $t$, i.e.
\[ \E\left[\sup_{t\in[0,1]}|\widehat{SCV}_t-SCV_t|\right]=\O\left((n/\log n)^{-\alpha/(4\alpha+2)}\right).\]

For estimation of the integrated covolatility we are not content with rate-optimality, but we also want to minimize the asymptotic variance. By independence we gain in efficiency by using on each block a convex combination of the estimators over all frequencies $j$. In order to estimate the integrated covolatility, we then just sum these estimators over all blocks. We end up with the following spectral estimation approach with local weights $w_{jk}$, satisfying $\sum_j w_{jk}=1$:
\begin{align}
\widehat{IC}_{w,n}^{(SPECV)}=\sum_{k=0}^{h^{-1}-1}h\sum_{j=1}^{nh} w_{jk}\|\Phi_{jk}\|_n^{-2}\left(\tilde x_{jk}\tilde y_{jk}-\eta_{XY}/n\right)\,.
\end{align}
The optimal weights (minimizing the variance) depend on the unknown spot volatility matrix. As will be shown in the proof of Theorem \ref{theo3}, they are given by $w_{jk}^{oracle}=w_j(\Sigma_{kh})\propto (\var(\tilde x_{jk}\tilde y_{jk}))^{-1}$ with
\begin{align}\label{woracle}
&w_j(\Sigma)
=\\
&\notag\frac{\left(\hspace*{-.05cm}\frac{\|\Phi_{jk}\|_n^{-4}}{n^2}(\eta_X^2\eta_Y^2\hspace*{-.01cm}+\hspace*{-.01cm}\eta_{XY}^2)\hspace*{-.05cm}+\hspace*{-.05cm}(1+\rho^2)(\sigma^X\sigma^Y)^2
\hspace*{-.05cm}+\hspace*{-.05cm}\frac{\|\Phi_{jk}\|_n^{-2}}{n}\hspace*{-.05cm}\big((\sigma^X)^2\eta_{Y}^2\hspace*{-.05cm}+\hspace*{-.05cm}(\sigma^Y)^2\eta_{X}^2\hspace*{-.05cm}+\hspace*{-.05cm}2\rho\sigma^X\sigma^Y\eta_{XY}\big)\right)^{-1}}
{\sum_{r=1}^{nh}\hspace*{-.1cm}\left(\frac{\|\Phi_{rk}\|_n^{-4}}{n^2}\hspace*{.05cm}\eta_X^2\eta_Y^2\hspace*{-.01cm}+\hspace*{-.01cm}\eta_{XY}^2\hspace*{-.05cm}+\hspace*{-.05cm}(1\hspace*{-.05cm}+\hspace*{-.05cm}\rho^2)(\sigma^X\sigma^Y)^2\hspace*{-.05cm}
+\hspace*{-.05cm}\frac{\|\Phi_{rk}\|_n^{-2}}{n}\hspace*{-.05cm}\big(\hspace*{-.05cm}(\sigma^X)^2\eta_{Y}^2\hspace*{-.05cm}+\hspace*{-.05cm}(\sigma^Y)^2\eta_{X}^2\hspace*{-.05cm}+\hspace*{-.05cm}2\rho\sigma^X\sigma^Y\eta_{XY}\big)\hspace*{-.05cm}\right)^{-1}}.
\end{align}
They give rise to the oracle version of our spectral estimator of covolatility (SPECV)
\begin{equation}\label{SPECVoracle}
\widehat{IC}_{oracle,n}^{(SPECV)}=\sum_{k=0}^{h^{-1}-1}h\sum_{j=1}^{nh} w_j(\Sigma_{kh})\|\Phi_{jk}\|_n^{-2}\left(\tilde x_{jk}\tilde y_{jk}-\eta_{XY}/n\right)\,.
\end{equation}

Using adequate consistent pilot estimates, we obtain a feasible estimator which is asymptotically as efficient as the oracle estimator. Besides \eqref{PilotSCV} we need the corresponding estimators for the spot volatilities $(\sigma_t^X)^2,(\sigma_t^Y)^2$:
\begin{equation}\label{PilotSV} \widehat{(\sigma_t^X)^2}=K^{-1}\sum_{k\in{\mathcal K}_t} \|\Phi_{1k}\|_n^{-2}\left(\tilde x_{1k}^2-\eta_X^2/n\right),\quad \widehat{(\sigma_t^Y)^2}=K^{-1}\sum_{k\in{\mathcal K}_t} \|\Phi_{1k}\|_n^{-2}\left(\tilde y_{1k}^2-\eta_Y^2/n\right),
\end{equation}
which also satisfy
\[ \E\left[\sup_{t\in[0,1]}\left(\big|\widehat{(\sigma_t^X)^2}-(\sigma^X_t)^2\big|+\big|\widehat{(\sigma_t^Y)^2}-(\sigma^X_Y)^2\big|\right) \right]=\O\left((n/\log n)^{-\alpha/(4\alpha+2)}\right)
\]
for $h\sim n^{-1/2}$, $K\sim n^{\alpha/(2\alpha+1)}$. In particular, all estimators are uniformly (in $t$) consistent provided the mesh $n^{-1}$ of the sample size tends to zero.

The resulting adaptive spectral estimator of covolatility (SPECV) for the integrated covolatility is
\begin{align}\label{SPECV}\widehat{IC}_n^{(SPECV)}=\sum_{k=0}^{h^{-1}-1}h\sum_{j=1}^{nh}
w_{j}\left(\widehat\Sigma_{\lfloor kh\rfloor_r}\right)\|\Phi_{jk}\|_n^{-2}\left(\tilde x_{jk}\tilde{y}_{jk}-\eta_{XY}/n\right)\end{align}
with the pilot estimator $\widehat\Sigma_{\lfloor t\rfloor_r},0\le r\le r^{-1}-1$ from \eqref{PilotSCV} and \eqref{PilotSV} for the spot covolatility matrix $\Sigma_t$, evaluated on coarser blocks of length $r>h$ with $r/h\in\mathds{N}$, inserted into the oracle weight formula \eqref{woracle}.

The piecewise constant pre-estimation step of the spot volatility matrix is performed from the same data set as the final estimator, but with respect to a coarser grid with block lengths $r\rightarrow 0$ such that \(\sqrt{r}/(n/\log n)^{-\alpha/(4\alpha+2)}\to\infty\), e.g. $r\thicksim n^{-1/4}$. This allows us to treat in the proof of \eqref{clt2} the dependence of pilot and final estimator by a tightness argument, which shows that the adaptive estimator is quite robust towards errors in the weights. Note also that this coarse grid approximation reduces the computational costs in the pre-estimation step. In the Appendix \ref{sec:app2} we also learn that high spectral frequencies have decreasing weights and exceeding some threshold will asymptotically not contribute to the estimation. For a tractable application of the SPECV it suffices to sum up frequencies in \eqref{SPECV} only up to a spectral cut-off $J_n\ll nh$. We refer to \cite{reiss} for more information on the cut-off value.

\begin{theo}\label{theo3}
We observe from model \eqref{E0} with $\Sigma \in C^{\alpha}(R),R>0,\alpha>1/2$ and $\underline\Sigma>0$. Choose $h\sim n^{-1/2}\log{(n)}$ and suppose that $\hat\Sigma_t$ is a uniformly consistent estimator of the spot volatility matrix in the sense that $\|\hat \Sigma-\Sigma\|_\infty=o_P(\delta_n)$ for some sequence $\delta_n\to 0$, e.g. based on \eqref{PilotSCV} and \eqref{PilotSV}. Choose a coarse block length $r$ with $r\to 0$ and $\sqrt{r}/\delta_n\to\infty$.
Then both, the adaptive and the oracle SPECV estimator, satisfy the same central limit theorem:
\begin{align}\label{clt}
n^{\nicefrac{1}{4}}\left(\widehat{IC}_{oracle,n}^{(SPECV)}-\int_0^1\rho_t\sigma_t^X\sigma_t^Y\,dt\right)&\rightsquigarrow \mathbf{N}\left(0,(\eta_X^2\eta_Y^2+\eta_{XY}^2)^{1/4}\int_0^1\mathfrak{v}_s\,ds\right),\\
\label{clt2}n^{\nicefrac{1}{4}}\left(\widehat{IC}_n^{(SPECV)}-\int_0^1\rho_t\sigma_t^X\sigma_t^Y\,dt\right)&\rightsquigarrow \mathbf{N}\left(0,(\eta_X^2\eta_Y^2+\eta_{XY}^2)^{1/4}\int_0^1\mathfrak{v}_s\,ds\right),
\end{align}
with local variance
\begin{align}\label{avarllrc}
\mathfrak{v}_t=\sqrt{2(A_t^2-B_t)B_t}\left(\sqrt{A_t+\sqrt{A_t^2-B_t}}-\operatorname{sgn}
\left(A_t^2-B_t\right)\sqrt{A_t-\sqrt{A_t^2-B_t}}\right)^{-1}\,
\end{align}
and $A_t=1/\sqrt{\eta_X^2\eta_Y^2+\eta_{XY}^2}\left(\eta_Y^2\left(\sigma_t^X\right)^2+\eta_X^2\left(\sigma_t^Y\right)^2+2\rho_t\sigma_t^X\sigma_t^Y\eta_{XY}\right)$ and $B_t=4\left(\sigma_t^X\sigma_t^Y\right)^2(1+\rho_t^2)$.
\end{theo}

We give a complete overview on the estimation of the (co)volatility matrix here by the according univariate estimator for the integrated volatilities:
\begin{align}\label{SPEV}\widehat{IV}_n^{(SPEV)}=\sum_{k=0}^{h^{-1}-1}h
\sum_{j=1}^{nh}w_{j}^X\left(\hat\sigma_{\lfloor kh\rfloor_r}^X\right)\|\Phi_{jk}\|_n^{-2}\left(\tilde x_{jk}^2-\eta_{X}^2/n\right)\,,\end{align}
which we call SPEV, with the oracle weights
\begin{align*}w^X_{j}\left(\sigma^X\right)=\frac{\left(\|\Phi_{jk}\|_n^{-2}(\eta_X^2/n)+(\sigma^X)^2\right)^{-2}}
{\sum_{l=1}^{nh}\left(\|\Phi_{lk}\|_n^{-2}(\eta_X^2/n)+(\sigma^X)^2\right)^{-2}}\end{align*}
and analogously for $\tilde Y$.

In general, the noise levels $\eta_X,\eta_Y,\eta_{XY}$ are unknown, but they can be estimated with faster rate $\sqrt{n}$, cf.\,Remark 1. A result with pre-estimated error covariance matrix $\widehat {\mathbf{H}}$ can be derived as for the pre-estimated $\Sigma_{kh}$ above. Furthermore, it is of high practical interest to study how our covolatility estimator behaves under vanishing microstructure noise level, i.e. in the case ${\mathbf H} =0$. In that case the oracle weights all equal $w_{jk}=1/(nh)$ and on each block we estimate the block covolatility by the sum
\[(nh)^{-1}\sum_{j=1}^{nh}\|\Phi_{jk}\|_n^{-2}\langle \Delta X,n\Phi_{jk}\rangle_n \langle \Delta Y,n\Phi_{jk}\rangle_n
\]
of discrete Fourier coefficients with respect to $(\Phi_{jk})_{1\le j\le nh}$. By Parseval identity this sum is equal to $n\langle \Delta X, \Delta Y\rangle_{nh;k}$. 
In conclusion, in the case ${\mathbf H}=0$ and for oracle weights our SPECV estimator reduces to the realized covolatility,  which is the natural estimator in this situation.

\section{Discussion and simulations\label{sec:4}}
\begin{figure}[t]
\fbox{\includegraphics[width=7cm]{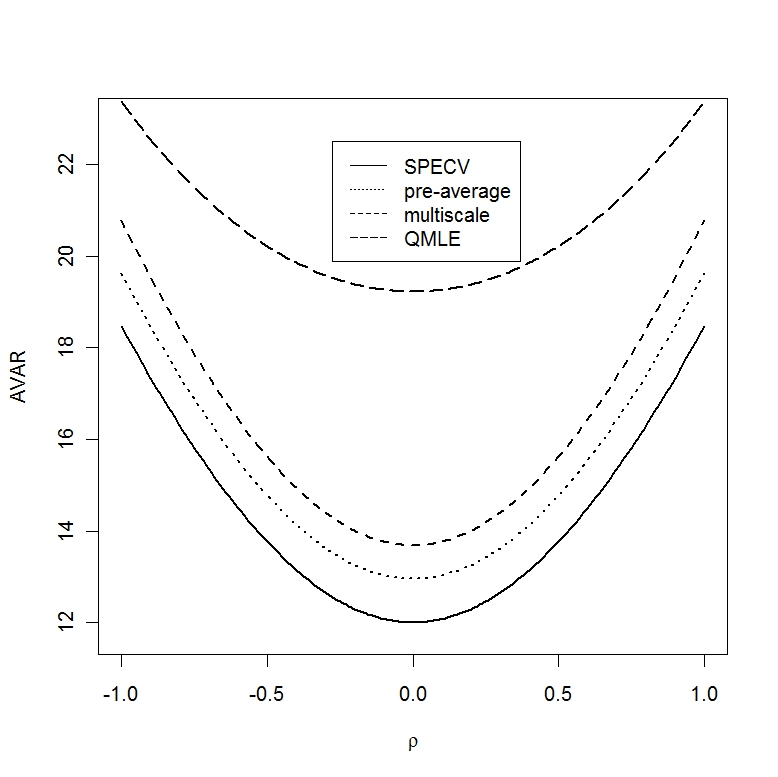}}~~~~\fbox{\includegraphics[width=7cm]{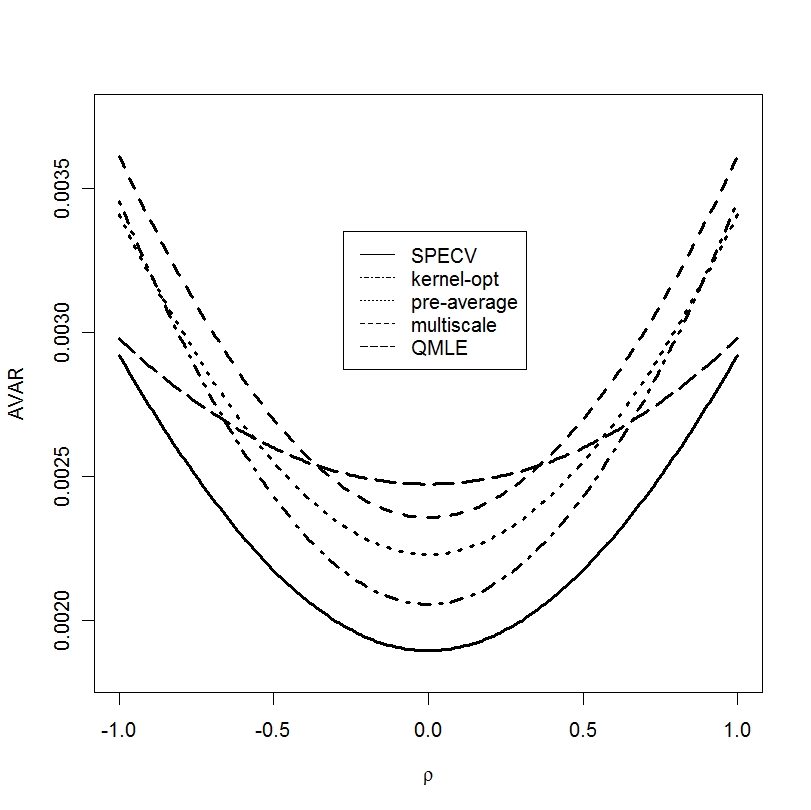}}\noindent
\caption{\label{fig:comp}Asymptotic variances of estimators for the covolatility in specific parametric case (left) and with time varying volatilities (right).}\end{figure}
Previously proposed nonparametric approaches have in common that they are quad\-ratic forms of the observation vectors and when choosing corresponding weights translate into each other and show accordant asymptotic properties. Nevertheless, the methods have been motivated from different points of view. The first two-scales realized volatility (TSRV) approach by \cite{zhangmykland} for the integrated volatility has been grounded on a subsampling method and a bias correction. Disregarding the bias correction, the subsampling estimator is the mean of lower frequent and hence less noise-sensitive realized volatilities. \cite{zhang} has extended this procedure to a linear combination using different time-scales (MSRV). The kernel approach by \cite{bn2} can be viewed as a linear combination of empirical autocovariances. Finally, the pre-average principle by \cite{JLMPV} pursuant to its name incorporates (pre-) averaged weighted observations on blocks. The latter is closest to our methodology, but using one Haar function and only one frequency instead of \eqref{Phi} combined with moving windows.\\
For all three the trade-off between the error due to noise and discretization is handled by choosing a global tuning parameter $c\sqrt{n}$, where $c$ is a constant, minimizing the RMSE to order $n^{-\nicefrac{1}{4}}$. Thus, the optimal convergence rate is attained. If we neglect in support of these methods the possible asymptotic influence of end effects, they have an asymptotic variance structure $\mathfrak{N}c^{-3}+\mathfrak{D}c+\mathfrak{C}c^{-1}$, where the signal part $\mathfrak{D}$ depends in our notation on $\Sigma$, the noise part $\mathfrak{N}$ on $\mathbf{H}$ and the cross term $\mathfrak{C}$ on both. Minimization leads to $c=\left(\left(-\mathfrak{C}+\sqrt{\mathfrak{C}^2+12\mathfrak{N}\mathfrak{D}}\right)/6\mathfrak{N}\right)^{-\nicefrac{1}{2}}$. The oracle solution is proportional to $\eta^{-1}$ for equal noise variances $\eta^2$ of $\tilde X$ and $\tilde Y$. Interestingly, \cite{bn2} have succeeded in the univariate case with constant volatility in approximately attaining the lower bound from \cite{gloter} by a clever selection method for their bandwidth and weights and also a feasible version with Tukey-Hanning kernels comes very close to that bound.
\begin{figure}[t]
\fbox{\includegraphics[width=7cm]{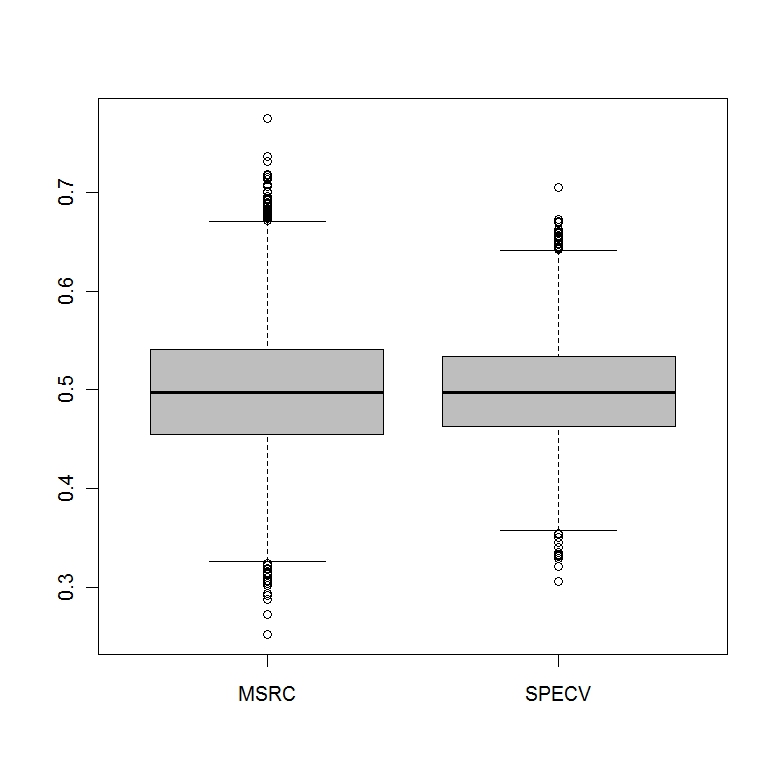}}~~~~\fbox{\includegraphics[width=7cm]{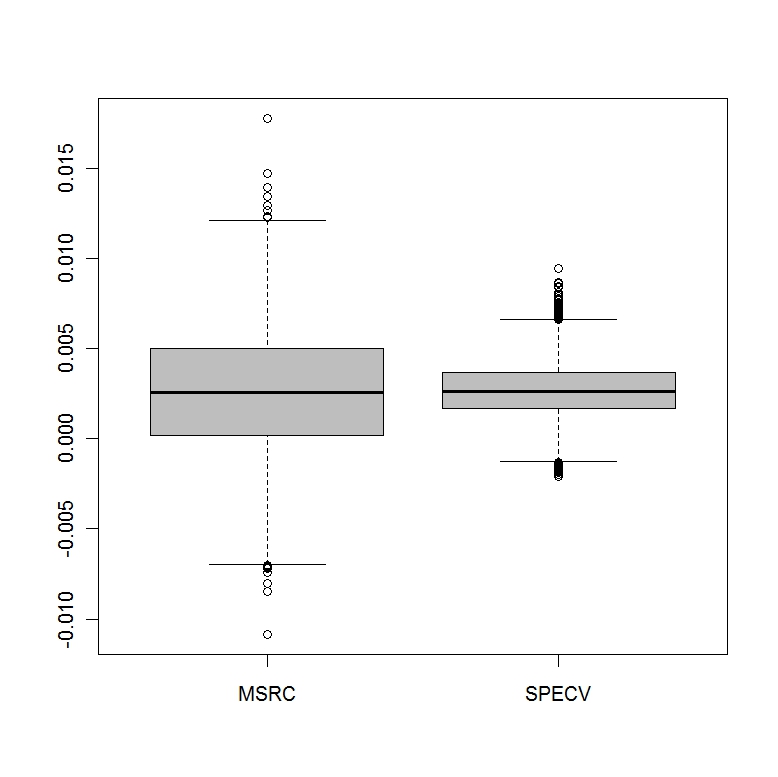}}
\fbox{\includegraphics[width=7cm]{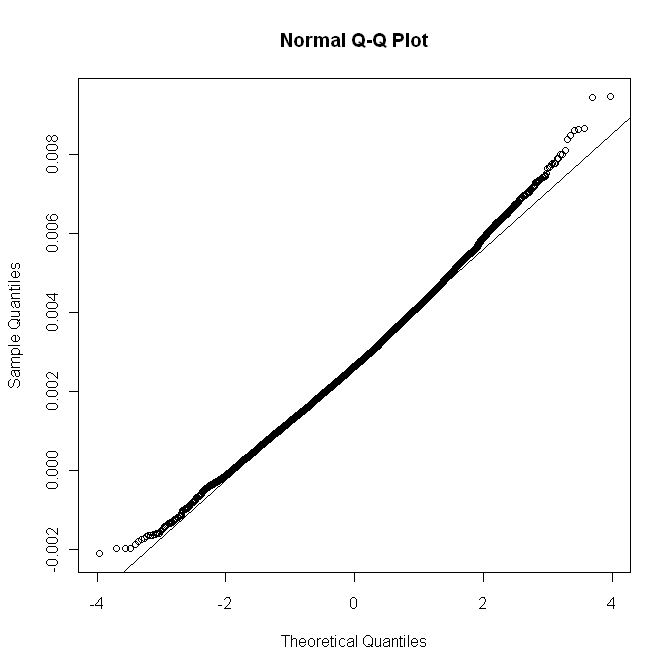}}~~~~\fbox{\includegraphics[width=7cm]{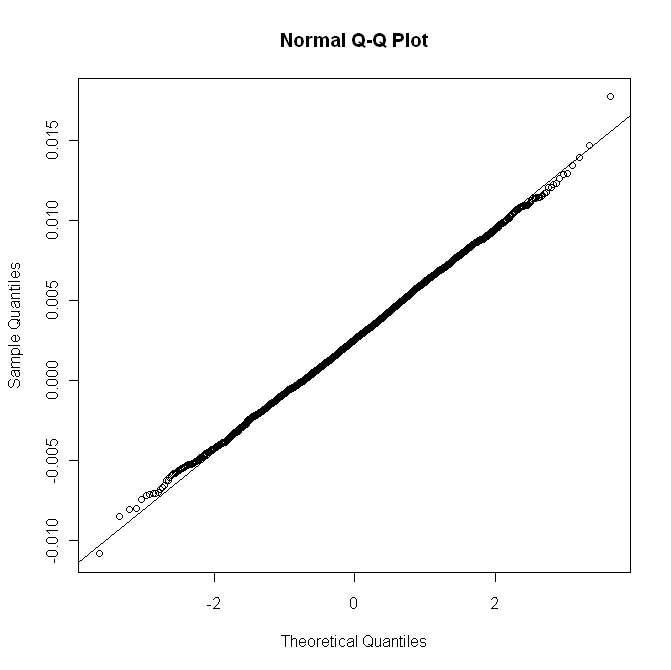}}\noindent
\caption{\label{fig:qqbpl}Boxplots for constant (top left) and time varying (top right) spot correlation and volatilities and normal QQ-Plots for MSRC (bottom left) and SPECV (bottom right) estimates in the time varying setting.}\end{figure}
Essentially, the main difference to our proposed approach is that we do not need to fix a tuning parameter and weights globally -- but are able to adapt weights locally dependent on the observations only on each particular block.\\
We content ourselves with the findings in an idealized statistical model which gives insight into the fundamental structure of the estimation problem. Note, that an i.\,i.\,d.\,assumption on the noise and Hölder-continuity conditions on the volatility processes are customary in the strand of literature on nonparametric estimation methods. In our opinion, it is convenient to look at methods derived from a simple model and inspect the effect of misspecification on them. In the microstructure noise setup, we might first think of a diffusion with constant parameters. \cite{xiu} has taken a path in this vein with reviving the classical MLE in this framework and proving its robustness to a typical nonparametric setup. A local parametric approach is more flexible and increases in general the performance. More surprising than the accordance of asymptotic properties for the aforementioned three nonparametric methods, is that \cite{xiu} reports that the Quasi-MLE approach is in this sense asymptotically equivalent to the kernel approach as well. This is not the case for our SPEV/SPECV approach what underlines the originality of our local spectral estimation method. Extensions of the theory that investigate the properties of the SPEV/SPECV in more general models, e.g. incorporating stochastic volatility and non-Gaussian errors, remain an open task for further research. For the moment the simple structure of SPEV/SPECV makes us confident that it can be robust to much more general model specifications.\\[1cm]
Let us give concrete examples to compare the asymptotic variances of our SPECV and the other methods. For the simple parametric setting with constant $\sigma^X=2$ and $\sigma^Y=1$ and $\eta\:=\eta_X=\eta_Y=1$, in Figure \ref{fig:comp} we depict the theoretical asymptotic variances for $\rho\in(-1,1)$ of the SPECV from Theorem \ref{theo3}, of the multi-scale realized covariance as deduced in \cite{bibinger2}, of the pre-average estimator as given in \cite{kinnepoldi} and of the QMLE from \cite{sahalia}, all with an optimal oracle tuning parameter selected as described above. The asymptotic variances are proportional to $\eta$, so that Figure \ref{fig:comp} rescaled by $\eta$ is meaningful for arbitrary noise levels. The SPECV has the smallest asymptotic variance and the QMLE the largest in this particular setup, all to the same optimal rate of convergence. The kernel method according to \cite{bn1} is not included, since the multivariate version has a non-optimal $n^{\nicefrac{1}{5}}$-rate by oversmoothing to the benefit of positive semi-definiteness. We stress that we intentionally have picked unequal constant volatilities here to the disadvantage of the QMLE which relies on the polarization identity. For equal volatilities $\sigma^X=\sigma^Y$ in our model $X+Y$ and $X-Y$ are independent and the polarized QMLE (also a polarized SPEV concurrent to the SPECV) will not suffer a disadvantage by polarization. Both, the QMLE and the SPECV exhibit asymptotic efficiency for $\rho=0$ in this setting. The approach presented in \cite{bn2} to derive asymptotic efficiency in the one-dimensional parametric case can easily be extended to a bivariate synchronous setting and renders a rate-optimal approach, asymptotically efficient for $\rho=0$ as well. Yet, none of these estimators is asymptotically efficient on the whole parameter space $(\rho,\sigma^X,\sigma^Y)\in\left((0,1),\mathds{R}_+,\mathds{R}_+\right)$ and it is beyond the scope of this article to finish the intricate quest for a globally multi-dimensional asymptotically efficient estimator, which can not be of the type of a smoothed realized covariance as all the above mentioned estimators. Note that in the one-dimensional setup the SPEV estimator \eqref{SPEV} by \cite{reiss} is even nonparametrically asymptotically efficient.\\ Here we aim at providing with SPECV a method that performs well for time varying functions by local adaptivity and thus focus on that setting in the following. For this purpose we compare asymptotic variances in the same spirit for $\rho\in(-1,1)$ and
\begin{align*}\sigma_t^X&=0.1-0.08\cdot\sin{\left(\pi t\right)},\,t\in[0,1]\,,\\
\sigma_t^Y&=0.15-0.07\cdot\sin{\left((6/7)\cdot  \pi t\right)},\,t\in[0,1]\,,\end{align*}
which will as well be considered for the simulation part below and captures the feature of higher volatilities at opening and closing. We add the theoretical asymptotic variance of a simple extension of the optimal kernel estimator for integrated volatility from \cite{bn2}, which can be approximated by Tukey-Hanning kernels. This approach features the smallest asymptotic variance in a wide domain of $\rho$ among the compared non-locally adaptive methods. Even so, the SPECV clearly comes below this benchmark. The right display of Figure \ref{fig:comp} shows that the gains of SPECV compared to the previously proposed methods are much more distinctively than in the parametric case. After this theoretical comparison and the conclusion that the SPECV is preferable, especially in the general nonparametric setting, we shed light on the finite sample size behaviour of our approach in a Monte Carlo study. \\[.2cm]
In the first simulation, we compare the SPECV with the multiscale realized covariance (MSRC), both with an oracle choice of weights and tuning parameter, respectively. First, we implement a simple parametric model with $n=30 000$ equidistant observations of $\tilde X$ and $\tilde Y$, where $\sigma^X=\sigma^Y=1,\rho=1/2$ and noise levels $\eta_X=\eta_Y=0.1$. The implemented MSRC as given in \cite{bibinger} is for synchronous observations a direct extension of the MSRV by \cite{zhang} and translates asymptotically to the kernel estimator with a cubic kernel. It is known to have a good finite sample size behavior. We implement the SPECV with an adequate heuristic choice $h=1/30$ such that $nh=1000$.\\
The empirical distribution of the estimates from 10 000 MC iterations are visualized in a boxplot in Figure \ref{fig:qqbpl}. The SPECV estimates have an empirical variance of $0.49\cdot  n^{-\nicefrac{1}{2}}$ and the MSRC of $0.71\cdot n^{-\nicefrac{1}{2}}$. The empirical finding is that in this setting the SPECV is closer to its theoretical asymptotic variance of about $0.46$ than the MSRC to its theoretical value of $0.52$.\\
Our main focus will be the time-varying nonparametric case. For an example of deterministic time-varying functions, set
\allowdisplaybreaks{
\begin{align*}\sigma_t^X&=0.1-0.08\cdot\sin{\left(\pi t\right)},\,t\in[0,1]\,,\\
\sigma_t^Y&=0.15-0.07\cdot\sin{\left((6/7)\cdot  \pi t\right)},\,t\in[0,1]\,,\\
\rho_t&=0.5+0.01\cdot\sin{\left( \pi t\right)},\,t\in[0,1]\,,\end{align*}}
where the volatilities are higher at the beginning and end of the observed interval and the correlation is only slowly varying, which mimics the basic realistic features. We keep the noise levels $\eta_X=\eta_Y=0.1$ fixed and rather high compared to the signal part. The known integrated covolatility equals 0.00269 here. Since the noise level is high and dominates the signal part, the frequencies chosen according to the above given selection rule for the MSRC estimator become large (over 1 000) and the computing time increases for these kind of nonparametric estimators. As can be seen in the right boxplot of Figure \ref{fig:qqbpl}, the SPECV outperforms the MSRC for non-constant volatilities and correlation more clearly. This confirms that the spectral local technique is more adequate to capture the effect of time-varying volatilities by local adaptation, not only theoretically but significantly in the finite sample case. The QQ-Plots in Figure \ref{fig:qqbpl} inspect the normal approximation for the two estimators from this Monte Carlo study in the time varying case.\\
We conclude the simulation study with an implementation of the adaptive SPEV/SPECV. We use piecewise constant pilot estimators \eqref{PilotSCV} and \eqref{PilotSV} for $\Sigma$ at times $l\cdot r=l\cdot 3h,\,l=0,\ldots,10$, each smoothed with $K=5$ adjacent blocks.\\
The 10 000 MC estimates of the adaptive SPECV are illustrated in Figure \ref{fig:bp}. Table \ref{stab1} summarizes the root mean squared errors of all three adaptive SPEV/SPECV estimators and the oracle SPECV and the oracle MSRC. The performance of the adaptive version of SPECV can not keep up with the oracle version, but in our simulation it is still slightly better than the oracle MSRC. For an adaptive MSRC the root mean squared error will clearly become larger and we refer to \cite{bibinger2} for the method and simulation results pertaining this point.
\begin{figure}
\centering
\renewcommand{\figurename}{Table}
\caption{\label{stab1}Comparison of root mean squared errors of MC (co-)volatility estimates.}
\renewcommand{\arraystretch}{2.25}
\parbox{2.5in}{\begin{tabular}{|c|c|}\hline
Estimator & \textbf{RMSE}\\
\hline
$\widehat{IV}_n^{(SPEV)}\,\mbox{for}\,\int_0^1\left(\sigma_t^X\right)^2\,dt$ & 0.0072\\
$\widehat{IV}_n^{(SPEV)}\,\mbox{for}\,\int_0^1\left(\sigma_t^Y\right)^2\,dt$ & 0.0086\\
$\widehat{IC}_n^{(SPECV)}\,\mbox{for}\,\int_0^1\rho_t\sigma_t^X\sigma_t^X\,dt$ & 0.0034\\
$\widehat{IC}_{oracle}^{(MSRC)}\,\mbox{for}\,\int_0^1\rho_t\sigma_t^X\sigma_t^X\,dt$ & 0.0035\\
$\widehat{IC}_{oracle,n}^{(SPECV)}\,\mbox{for}\,\int_0^1\rho_t\sigma_t^X\sigma_t^X\,dt$ & 0.0015\\
\hline
\end{tabular}}
\begin{minipage}{2.75in}
\includegraphics[width=7.00cm]{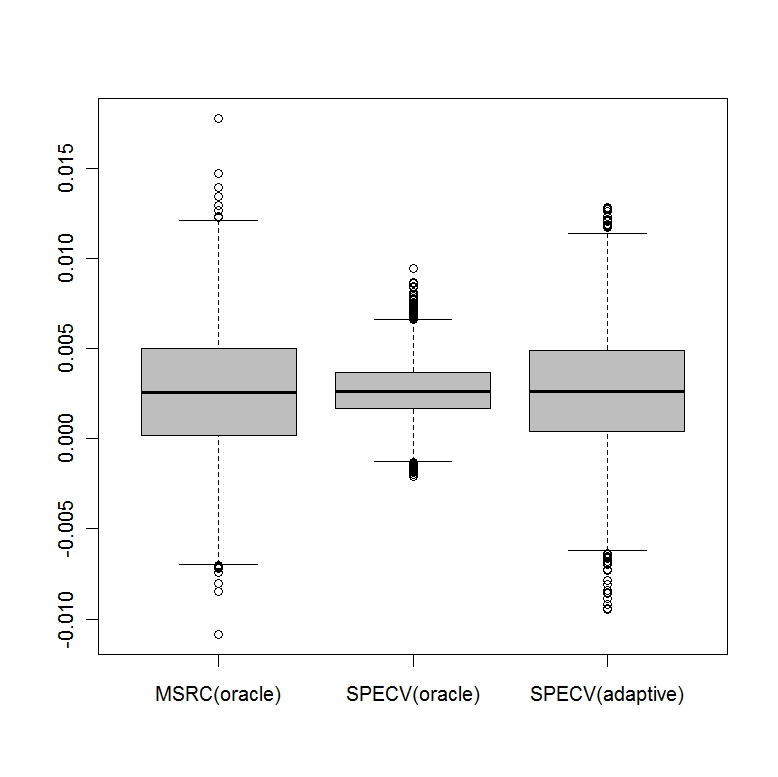}
\end{minipage}
\renewcommand{\figurename}{Figure}
\caption{\label{fig:bp}Boxplot of 10000 MC iterations of the oracle MSRC and oracle/adaptive SPECV.}
\end{figure}

\newpage
\appendix
\section{Appendix: Proofs of asymptotic equivalence\label{sec:app1}}
\subsection*{Proof of Theorem \ref{theo1}}

We start with the constructive proof that \eqref{E0} is at least as informative as \eqref{E1}.
We use the linear B-splines
\begin{align*}b_i(t)=\1_{[\frac{i-1}{n}\,,\,\frac{i+1}{n}]}(t)\min{\left(1+n\left(t-\frac{i}{n}\right)\,,\,1-n\left(t-\frac{i}{n}\right)\right)}\,,\end{align*}
i.\,e.\,$\operatorname{supp} b_i=\left[(i-1)/n,(i+1)/n\right], b_i(i/n)=1$, and $b_i$ linear on $\left[(i-1)/n,i/n\right]$ and $\left[i/n,(i+1)/n\right]$. Consider the centred Gaussian process $\hat{\mathbf{Z}}$ defined by
\begin{align}\hat{\mathbf{Z}}_t=\sum_{i=1}^{n}\tilde {\mathbf{Z}}_i b_i(t)=\sum_{i=1}^n\mathbf{Z}_{\frac{i}{n}}b_i(t)+\sum_{i=1}^n\mathbf{\eps}_i b_i(t)\,.\end{align}
The covariance function of $\hat{\mathbf{Z}}$ is
\begin{align*}\E\left[\hat{\mathbf{Z}}_t{\hat{\mathbf{Z}}_s}^{\top}\right]=\sum_{i,j=1}^n\mathbf{A}\left(\frac{i\wedge j}{n}\right)b_i(t)b_j(s)+\mathbf{H}\sum_{i=1}^n b_i(t)b_i(s)\end{align*}
with
\begin{align*}\mathbf{A}(t)\:=\int_0^t\Sigma_s\,ds=\left(\begin{array}{cc} \int_0^t(\sigma_s^X)^2\,ds & \int_0^t\rho_s\sigma_s^X\sigma_s^Y\,ds\\ \int_0^t\rho_s\sigma_s^X\sigma_s^Y\,ds& \int_0^t(\sigma_s^Y)^2\,ds\end{array}\right)~\mbox{and}~\mathbf{H}=\left(\begin{array}{cc} \eta_X^2&\eta_{XY}\\ \eta_{XY}&\eta_Y^2\end{array}\right)\,.\end{align*}
For any $\mathbf{f}=\left(f_X,f_Y\right)^{\top} \in L^2\left([0,1],\mathds{R}^2\right)$, we have
\begin{align*}
&\E\left[\langle \mathbf{f},\hat{\mathbf{Z}}\rangle^2\right]=\E\left[\left(\langle f_X,\hat X\rangle+\langle f_Y,\hat Y\rangle\right)^2\right]\\
&~=\E\left[\langle f_X,\hat X\rangle^2\right]+\E\left[\langle f_Y,\hat Y\rangle^2\right]+2\,\E\left[\langle f_X,\hat X\rangle \langle f_Y,\hat Y\rangle\right]\\
&~=\sum_{i,j=1}^nA_{11}\left(\frac{i\wedge j}{n}\right)\langle f_X,b_i\rangle\langle f_X,b_j\rangle+\sum_{i=1}^n\eta_X^2\langle f_X,b_i\rangle^2\\
&~~+\sum_{i,j=1}^nA_{22}\left(\frac{i\wedge j}{n}\right)\langle f_Y,b_i\rangle\langle f_Y,b_j\rangle+\sum_{i=1}^n\eta_Y^2\langle f_Y,b_i\rangle^2\\
&~~+2\,\sum_{i,j=1}^nA_{12}\left(\frac{i\wedge j}{n}\right)\langle f_X,b_i\rangle\langle f_Y,b_j\rangle+2\sum_{i=1}^n\eta_{XY}\langle f_X,b_i\rangle\langle f_Y,b_i\rangle\,.
\end{align*}
The sum of the three terms induced by the observation noise is bounded from above by $n^{-1}(\eta_X^2\|f_X\|^2+\eta_Y^2\|f_Y\|^2+2\,\eta_{XY}\langle f_X,f_Y\rangle)$, since
\begin{align*}&~~~\eta_X^2\sum_{i=1}^n\langle f_X,b_i\rangle^2+\eta_Y^2\sum_{i=1}^n\langle f_Y,b_i\rangle^2+2\eta_{XY}\sum_{i=1}^n\langle f_X,b_i\rangle \langle f_Y,b_i\rangle\\
&~~~=\left(\frac{1}{2}+\frac{\eta_{XY}}{2\eta_X\eta_Y}\right)\sum_{i=1}^n\langle \eta_X f_X+\eta_Y f_Y,b_i\rangle^2+\left(\frac{1}{2}-\frac{\eta_{XY}}{2\eta_X\eta_Y}\right)\sum_{i=1}^n\langle \eta_X f_X-\eta_Y f_Y,b_i\rangle^2\\
&~~~\le \left(\frac{1}{2}+\frac{\eta_{XY}}{2\eta_X\eta_Y}\right)n^{-1}\|\eta_Xf_X+\eta_Yf_Y\|^2+\left(\frac{1}{2}-\frac{\eta_{XY}}{2\eta_X\eta_Y}\right)n^{-1}\|\eta_Xf_X-\eta_Yf_Y\|^2\\
&~~~=n^{-1}\left(\eta_X^2\|f_X\|^2+\eta_Y^2\|f_Y\|^2+2\,\eta_{XY}\langle f_X,f_Y\rangle\right)\,.
\end{align*}
For the upper bound we have used that $\int_0^1nb_i(t)\,dt=1$ implies $\langle f_X,n b_i\rangle^2\le\langle f_X^2,nb_i\rangle$ by Jensen's inequality and $\sum_i b_i\le 1$ and analogously for the other terms. Now observe that $\E\left[\langle f,\mathbf{H}\,d\mathbf{W}\rangle\right]=\E\left[\int f_t^{\top}\mathbf{H}\,d\mathbf{W}_t\right]=(\eta_X^2\|f_X\|^2+\eta_Y^2\|f_Y\|^2+2\,\eta_{XY}\langle f_X,f_Y\rangle)$.\\
As a consequence, observations from $\bar{\mathbf{Z}}$ defined by
\begin{align}\label{iwn}d\bar{\mathbf{Z}}=\sum_{i=1}^n\mathbf{Z}_{\frac{i}{n}}b_i(t)\,dt+\frac{1}{\sqrt{n}}\mathbf{H}^{\nicefrac{1}{2}}\,d\mathbf{W}_t\end{align}
with a two-dimensional standard Brownian motion $\mathbf{W}$ can be generated from \eqref{E0} by adding additional $\mathbf{N}\left(0,\bar C-\hat C\right)$-noise, where $\hat C:L^2\rightarrow L^2$ is the covariance operator of $\hat Z$ and the covariance operator $\bar C:L^2\rightarrow L^2$ associated with \eqref{iwn} is given by
\begin{align*}\bar C \mathbf{f}(t)=\sum_{i,j=1}^n\mathbf{A}\left(\frac{i\wedge j}{n}\right)\langle f,b_j\rangle b_i(t)+n^{-1}\mathbf{H} \mathbf{f}(t)\,,\mathbf{f}\in L^2\left([0,1],\mathds{R}^2\right)\,.\end{align*}
Let $C$ be the covariance operator
\begin{align*}C\mathbf{f}(t)=\int_0^1\left(\int_0^{t\wedge u}\mathbf{A}(s)\,ds\right)\mathbf{f}(u)\,du+n^{-1}\mathbf{H}\mathbf{f}(t)\end{align*}
from \eqref{E1}. In the following $\underline{\mathbf{H}}>0$ denotes the smallest eigenvalue of $\mathbf{H}$ as in Section \ref{sec:2}.
In the extension of the findings for the one-dimensional case, which has been treated in Section \textbf{A.2} in \cite{reiss}, we make use of the convenient upper bound for the squared Hellinger distance between two normal measures by the squared Hilbert-Schmidt norm denoted $\|\,\cdot\,\|_{\text{HS}}$. For a concise introduction on Hellinger distances between Gaussian measures and the Hilbert-Schmidt norm we refer to Section \textbf{A.1} in \cite{reiss}.\\
The asymptotic equivalence of observing $\bar{\mathbf{Z}}$ and $\tilde {\mathbf{Z}}$ in \eqref{E1} is ensured by the Hellinger distance bound
\begin{align*}
\operatorname{H}^2\left(\mathcal{L}\left(\bar{\mathbf{Z}}\right)\,,\,\mathcal{L}\left(\tilde{\mathbf{Z}}\right)\right)&\le 2\,\|C^{-\nicefrac{1}{2}}\left(\bar C-C\right)C^{-\nicefrac{1}{2}}\|^2_{\text{HS}}\\
&\le 2\,\underline{\mathbf{H}}^{-2} n^2\int_0^1\int_0^1\left\|\mathbf{A}(t\wedge s)-\sum_{i,j=1}^n\mathbf{A}\left(\frac{i\wedge j}{n}\right)b_i(t)b_j(s)\right\|^2\,dt\,ds\\
&=\mathcal{O}\left(\underline{\mathbf{H}}^{-2}R^2n^{-(2\alpha\wedge 1)}\right)=\KLEINO(1)~\text{for}~\alpha>0\,.\end{align*}
Note that we have estimated the $L^2$-distance between $\mathbf{A}(t\wedge s)$ and its coordinate-wise linear interpolation by $\mathcal{O}(n^{-1-\alpha})$ using a standard approximation result based on the fact that the function $(t,s)\mapsto \mathbf{A}(t\wedge s)$ lies in the class $C^{1+\alpha}$ away from the diagonal $\{t=s\}$   due to $\mathbf{A}'(t)=\Sigma_t\in C^\alpha$ and is Lipschitz at the diagonal (on the $n-1$ squares $[(i-1)/n,(i+1)/n]$ the pointwise bound $\mathcal{O}(n^{-1})$ only contributes $(n-1)\mathcal{O}(n^{-2})=\mathcal{O}(n^{-1})$ to the squared $L^2$-distance).

The proof that \eqref{E1} is at least as informative as \eqref{E0} is obtained by a similar estimate and a generalization of the construction technique from the one-dimensional setting. For this purpose, set
\begin{align*}\mathbf{Z}^{\prime}_i=n\int_{(2i-1)/2n}^{(2i+1)/2n}\,d\tilde{\mathbf{Z}}_t=n\int_{(2i-1)/2n}^{(2i+1)/2n}\mathbf{Z}_t\,dt+\mathbf{\eps}_i,\,1\le i\le (n-1),\\
\mathbf{Z}^{\prime}_n=2\,n\int_{(2n-1)/2n}^{1}\,d\tilde{\mathbf{Z}}_t=2\,n\int_{(2n-1)/2n}^{1}\mathbf{Z}_t\,dt+\mathbf{\eps}_n,\end{align*}
with
\begin{align*}\mathbf{\eps}_i=\sqrt{n}\int_{(2i-1)/2n}^{(2i+1)/2n}\mathbf{H}^{\nicefrac{1}{2}}\,d\mathbf{W}_t\sim\mathbf{N}(0,\mathbf{H})\,.\end{align*}
The estimate that
\begin{align*}
&\operatorname{H}^2\left(\mathcal{L}\left({\mathbf{Z}_1^{\prime}},\ldots,{\mathbf{Z}_n^{\prime}}\right)\,,\,\mathcal{L}\left(\tilde{\mathbf{Z}}_1,\ldots,\tilde{\mathbf{Z}}_n\right)\right)\le 2\,\|\tilde{C}^{-\nicefrac{1}{2}}\left(C^{\prime}-\tilde{C}\right)\tilde{C}^{-\nicefrac{1}{2}}\|^2_{\text{HS}}\\
&~~\le 2\,\underline{\mathbf{H}}^{-2}\|C^{\prime}-\tilde{C}\|^2_{\text{HS}}=\mathcal{O}\left(\underline{\mathbf{H}}^{-2}R^2n^{-2\alpha}\right)
\end{align*}
establishes the result. Altogether, the Le Cam distance between the experiments \eqref{E0} and \eqref{E1} is of order $\mathcal{O}\left(\underline{\mathbf{H}}^{-1}Rn^{-\alpha}\right)$. Assuming that $\mathbf{A}$ is $(1+\alpha)$-Hölder continuous ($\alpha$-Hölder regularity of the covolatility and volatilities), the asymptotic equivalence of the statistical experiments with discretely observed noisy diffusions and the continuous time white noise model is deduced. $\hfill\Box$

\subsection*{Proof of Theorem \ref{theo2}}

The proof affiliates to the one-dimensional result and its proof in Section \textbf{A.3} of \cite{reiss}. It is shown that the Hilbert-Schmidt norm of the difference between the experiment \eqref{E1} and the one where $\Sigma$ is evaluated at times $\lfloor t\rfloor_h\:=\min{\{k\,h|k\,h\le t\}},\,1\le k\le h^{-1}-1$ tends to zero.\\
In the two-dimensional setting, we have a Hölder bound
\begin{align*}\|\Sigma_t-\Sigma_{\lfloor t\rfloor_h}\|_{\infty}\le Rh^{\alpha}\,,\,t\in[0,1]\,.\end{align*}
Denote $C_{\Sigma}$ the covariance operator associated with the experiment \eqref{E1} with volatility matrix $\Sigma$. For $f\in L^2\left([0,1],\mathds{R}^2\right)$ let $F:[0,1]\rightarrow \mathds{R}^2$ be the corresponding antiderivative with $F(1)=(0,0)^{\top}$. The difference of the two covariance operators of experiments with $\Sigma$ and $\Sigma^h$ where $\Sigma_t^h\:=\Sigma_{\lfloor t\rfloor_h}$, respectively, pertains only the signal part:
\begin{align*}\langle (C_{\Sigma}-C_{\Sigma^h})f,f\rangle&=\int_0^1F_t^{\top}\left(\Sigma_t-\Sigma_t^h\right)F_t\,dt\le \|\Sigma-\Sigma^h\|_{\infty}\langle \mathfrak{C}f,f\rangle
\end{align*}
by partial integration, where $\mathfrak{C}$ denotes the covariance operator of a standard two-dimensional Brownian motion. We end up with the following upper bound for the Hilbert-Schmidt norm:
\begin{align*}\|C_{\Sigma}^{-\nicefrac{1}{2}}(C_{\Sigma^h}-C_{\Sigma})C_{\Sigma}^{-\nicefrac{1}{2}}\|_{\text{HS}}&\le \|\Sigma-\Sigma^h\|_{\infty}\|C_{\Sigma}^{-\nicefrac{1}{2}}\,\mathfrak{C}\,C_{\Sigma}^{-\nicefrac{1}{2}}\|_{\text{HS}}\\
&\le \|\Sigma-\Sigma^h\|_{\infty}\left\|\left(\mathfrak{C}\underline{\Sigma}+\underline{\mathbf{H}}n^{-1}\operatorname{id}\right)^{-\nicefrac{1}{2}}\mathfrak{C}\left(\mathfrak{C}\underline{\Sigma}+\underline{\mathbf{H}}n^{-1}\operatorname{id}\right)^{-\nicefrac{1}{2}}\right\|_{\text{HS}}\\
&\le R h^{\alpha}\|G( \mathfrak{C})\|_{\text{HS}}\,.\end{align*}
The function $G(z)=z\left(z\underline{\Sigma}+\underline{\mathbf{H}}n^{-1}\right)^{-1}$ is applied to $\mathfrak{C}$ employing functional calculus.\\
The operator $\mathfrak{C}$ has the same spectral values as the covariance operator of a one-dimensional standard Brownian motion with double multiplicity. Hence, the result is derived directly from the spectral analysis for the one-dimensional case in \cite{reiss}.
$\hfill\Box$

\subsection{Proof of Proposition \ref{propE3}}

The proof follows exactly along the lines of proof for Theorem \ref{theo1}. The only difference is the bound on the $L^2([0,1]^2)$-distance between the functions $\mathbf{A}(t\wedge s)$ and $\sum_{i,j}\mathbf{A}((i\wedge j)/n)b_i(t)b_j(s)$. Since $\Sigma$ is block-wise constant, $\mathbf A$ is linear on each interval $[(i-1)/n,i/n]$. By the linear interpolation property the two functions coincide on each square $[(i-1)/n,i/n]\times[(j-1)/n,j/n]$ for $i\not= j$. For the $n$ squares where $i=j$, the Lipschitz property of $(t,s)\mapsto {\mathbf A}(t\wedge s)$ yields a total $L^2$-distance of order $n^{-3/2}$ (cf. again proof of Theorem \ref{theo1}) and the bound on the Le Cam distance follows.

\section{Appendix: Asymptotics of the local spectral (co-)volatility estimator\label{sec:app2}}
We start with the following standard formula for a bivariate normal distribution which will be used implicitly several times.
\begin{lem}
For a Gaussian random vector
\begin{align*}\left(\begin{array}{c}X\\ Y\end{array}\right)\sim\mathbf{N}\left(\mathbf{0}\,,\left(\begin{array}{cc}\sigma_X^2&\rho\sigma_X\sigma_Y\\ \rho\sigma_X\sigma_Y& \sigma_Y^2\end{array}\right)\right)\end{align*}
it holds true that \begin{align}\label{nv}\var(X^2Y^2)=(1+\rho^2)\sigma_X^2\sigma_Y^2~~\mbox{and}~~\var(X^2)=2\sigma_X^4,\,\var(Y^2)=2\sigma_Y^4\,.\end{align}
\end{lem}
\begin{proof}
The nature of the Gaussian distribution allows us to write $Y=\rho(\sigma_Y/\sigma_X)X+\sqrt{1-\rho^2}\sigma_Y Z$ where $Z\sim\mathbf{N}(0,1)$ is independent of $X$. Since $\E\left[X^4\right]=3\sigma_X^4$ and
\begin{align*}\E\left[X^2Y^2\right]=\E\left[X^2\E\left[Y^2|X^2\right]\right]&=\E\left[X^2\left(\rho^2(\sigma_Y^2/\sigma_X^2)X^2+(1-\rho^2)\sigma_Y^2\E\left[Z^2\right]\right)\right]\\
&=\rho^2(\sigma_Y^2/\sigma_X^2)\E\left[X^4\right]+(1-\rho^2)\E\left[X^2\right]\sigma_Y^2=(1+2\rho^2)\sigma_X^2\sigma_Y^2\,,\end{align*}
we directly conclude the statement of the Lemma.
\end{proof}
The elementary identities \eqref{nv} and \eqref{sbp} are central tools in the error analysis for the SPECV-estimator. The latter is proved in the following.
\subsection*{Proof of Proposition \ref{propsbp}}
Equation \eqref{sbp} is basically an application of the discrete summation by parts analogue to the integration by parts formula, also called Abel transformation.\\ The elementary identity $\sin{(x+h)}-\sin{x}=2\cos{(x+h/2)}\sin{(h/2)}$ yields:
\begin{align*}\sum_{l=1}^n\Delta \tilde Y_l\Phi_{jk}\left(\frac{l}{n}\right)&=-\sum_{l=1}^{n-1}\tilde Y_{\frac{l}{n}}\left(\Phi_{jk}\left(\frac{l+1}{n}\right)-\Phi_{jk}\left(\frac{l}{n}\right)\right)+\tilde Y_1\Phi_{jk}\left(1\right)-\tilde Y_0\Phi_{jk}\left(n^{-1}\right)\\
&=-\sum_{l=0}^{n-1}\tilde Y_{\frac{l}{n}}\left(\Phi_{jk}\left(\frac{l+1}{n}\right)-\Phi_{jk}\left(\frac{l}{n}\right)\right)\\
&=-\sum_{l=0}^{n-1}\tilde Y_{\frac{l}{n}}\;\varphi_{jk}\left(\frac{l+\frac{1}{2}}{n}\right)\frac{1}{n}\,.\end{align*}
The boundary terms vanish due to $\Phi_{jk}(0)=\Phi_{jk}(1)=0$. Further simple relations for trigonometric functions reveal the orthogonality properties \eqref{o1} and \eqref{o2}.
Without loss of generality, consider the first block $k=0$:
\begin{align*}
\left[\varphi_{j0},\varphi_{r0}\right]_n&=\frac{1}{n}\sum_{l=0}^{nh-1}\frac{2}{h}\cos{\left(j\pi h^{-1}n^{-1}(l+1/2)\right)}\cos{\left(r\pi h^{-1}n^{-1}(l+1/2)\right)}\\
&=\frac{1}{n}\sum_{l=0}^{nh-1}h^{-1}\left(\cos{\left((j+r)\pi h^{-1}n^{-1}(l+1/2)\right)}+\cos{\left((j-r)\pi h^{-1}n^{-1}(l+1/2)\right)}\right)\\
&=\delta_{jr}\,.\end{align*}
The last equality holds since for arbitrary $m\in\mathds{N}$:
\begin{align*}\sum_{l=0}^{nh-1}\cos{\left(\frac{m\pi}{hn}\left(l+\frac{1}{2}\right)\right)}
&=\hspace*{-.1cm}\sum_{l=0}^{\lfloor\frac{nh-1}{2}\rfloor}\sin{\left(\pi\left(\frac{2l+1}{2}\frac{m}{hn}+\frac{1}{2}\right)\right)}\hspace*{-.05cm}+\hspace*{-.1cm}\sum_{l=0}^{\lfloor\frac{nh-1}{2}\rfloor}\sin{\left(\pi\left(\frac{1}{2}+m-\frac{2l+1}{2}\frac{m}{hn}\right)\right)}\\
&=\hspace*{-.1cm}\sum_{l=0}^{\lfloor\frac{nh-1}{2}\rfloor}\sin{\left(\pi\left(\frac{2l+1}{2}\frac{m}{hn}+\frac{1}{2}\right)\right)}\hspace*{-.05cm}-\hspace*{-.1cm}\sum_{l=0}^{\lfloor\frac{nh-1}{2}\rfloor}\sin{\left(\pi\left(\frac{1}{2}+\frac{2l+1}{2}\frac{m}{hn}\right)\right)}=0\,.\end{align*}
Analogously we deduce that
\begin{align*}
\langle \Phi_{j0},\Phi_{r0}\rangle_n&=\frac{1}{n}\sum_{l=1}^{nh}\frac{\sin{\left(j\pi h^{-1}n^{-1}l\right)}\sin{\left(r\pi h^{-1}n^{-1}l\right)}}{2hn^2\sin{\left(\frac{j\pi}{2nh}\right)}\sin{\left(\frac{r\pi}{2nh}\right)}}\\
&=\frac{1}{n}\sum_{l=1}^{nh}\frac{\cos{\left((j-r)\pi h^{-1}n^{-1}l\right)}-\cos{\left((j+r)\pi h^{-1}n^{-1}l\right)}}{4hn^2\sin{\left(\frac{j\pi}{2nh}\right)}\sin{\left(\frac{r\pi}{2nh}\right)}}\\
&=\delta_{jr}\left(4n^2\sin^2{(j\pi/(2nh))}\right)^{-1}=\delta_{jr}\|\Phi_{j0}\|_n^2\,.\end{align*}
We conclude that the families of functions $(\varphi_{jk}),(\Phi_{jk})$ are orthogonal systems with respect to $\left[\cdot,\cdot\right]_n$ and $\langle\cdot,\cdot\rangle_n$, respectively.\hfill $\Box$\\[.2cm]

\subsection*{Proof of Theorem \ref{theo3}}

Though we have exploited the well-known distribution characteristics for blockwise averages $(\tilde x_{jk},\tilde y_{jk})$ directly in our simple Gaussian model in order to motivate the spectral estimator, for a better transparency and clarity, we give a detailed analysis for the asymptotic expectation and variance here. By Proposition \ref{propE3} we can equivalently work with model $({\mathcal E}_3)$, where the observations are generated by a blockwise constant spot volatility $\Sigma_{\floor{t}}$.\\
Consider at first $\widehat{IC}_{oracle,n}^{(SPECV)}$ with known spot volatility matrix and correlation. We drop the superscript and subscripts in the following. The estimator is (asymptotically in \eqref{E0}) unbiased since
\begin{align*}\E\left[\widehat{IC}\right]&=\sum_{k=0}^{h^{-1}-1}h\sum_{j=1}^{nh} \|\Phi_{jk}\|_n^{-2}w_{jk}\,\E\left[\langle \Delta X,n\Phi_{jk}\rangle_n\langle \Delta Y,n\Phi_{jk}\rangle_n+\left[\eps^X,\varphi_{jk}\right]_n\left[\eps^Y,\varphi_{jk}\right]_n-\frac{\eta_{XY}}{n}\right]\\
&=\sum_{k=0}^{h^{-1}-1}h\sum_{j=1}^{nh}\|\Phi_{jk}\|_n^{-2}w_{jk}\left(\E\left[n\langle\Delta X,\Delta Y\rangle_{nh;k}\right]\|\Phi_{jk}\|^2_n+\frac{\eta_{XY}}{n}\left(\left[\varphi_{jk},\varphi_{jk}\right]_n-1\right)\right)\\
&=\sum_{k=0}^{h^{-1}-1}h\rho_{kh}\sigma_{kh}^X\sigma_{kh}^Y=\int_0^1\rho_t\sigma_t^X\sigma_t^Y\,dt+\KLEINO(1)\,,\end{align*}
in view of Parseval identity, 
It\^{o} isometry, the orthogonality relations \eqref{o1} and \eqref{o2} and $\sum_{j=1}^{nh}w_{jk}=1$.\\
The variance calculation is simplified by the independent block structure:
\begin{align*}
\var\left(\widehat{IC}\right)=\sum_{k=0}^{h^{-1}-1}h^2\var\left(\sum_{j=1}^{nh}\|\Phi_{jk}\|_n^{-2}w_{jk}\tilde x_{jk}\tilde y_{jk}\right)\,.\end{align*}
Consider the variance on the $k$th block. By the orthogonality relations \eqref{o1} and \eqref{o2} of the $\varphi_{jk}$'s and $\Phi_{jk}$'s and application of \eqref{nv} to $\Sigma_{kh}$ and $\mathbf{H}$, the evaluation of the variance on the block yields
\allowdisplaybreaks[3]{
\begin{align*}&\var\left(\sum_{j=1}^{nh}\|\Phi_{jk}\|_n^{-2}w_{jk}\tilde x_{jk}\tilde y_{jk}\right)=\sum_{j=1}^{nh}\|\Phi_{jk}\|_n^{-4}w_{jk}^2\var\left(\left[\eps^X,\varphi_{jk}\right]_n\left[\eps^Y,\varphi_{jk}\right]_n\right)\\
&\hspace*{1.5cm}+\sum_{j=1}^{nh}\|\Phi_{jk}\|_n^{-4}w_{jk}^2\left(\E\left[\langle \Delta X,n\Phi_{jk}\rangle_n^2\langle \Delta Y,n\Phi_{jk}\rangle_n^2\right]-\left(\E\left[\langle \Delta X,n\Phi_{jk}\rangle_n\langle \Delta Y,n\Phi_{jk}\rangle_n\right]\right)^2\right)\\
&\hspace*{1.5cm}+\sum_{j=1}^{nh}\|\Phi_{jk}\|_n^{-4}w_{jk}^2\left(\E\left[\left[\eps^X,\varphi_{jk}\right]_n^2\langle \Delta Y,n\Phi_{jk}\rangle_n^2\right]+\E\left[\left[\eps^Y,\varphi_{jk}\right]_n^2\langle \Delta X,n\Phi_{jk}\rangle_n^2\right]\right)\\
&\hspace*{1.5cm}+2\sum_{j=1}^{nh}\|\Phi_{jk}\|_n^{-4}w_{jk}^2\E\left[\left(\left[\eps^X,\varphi_{jk}\right]_n\left[\eps^Y,\varphi_{jk}\right]_n\langle\Delta X,n\Phi_{jk}\rangle_n\langle\Delta Y,n\Phi_{jk}\rangle_n\right)\right]\\
&=\sum_{j=1}^{nh}\|\Phi_{jk}\|_n^{-4}w_{jk}^2\left(\frac{\eta_X^2\eta_Y^2+\eta_{XY}^2}{n^2}\left(\left[\varphi_{jk},\varphi_{jk}\right]_n\right)^2+\E\left[n^2(\langle\Delta X,\Delta Y\rangle_{nh;k})^2\right]\|\Phi_{jk}\|_n^4\right.\\
&\left.\hspace*{1cm}+\frac{\eta_X^2}{n}\left[\varphi_{jk},\varphi_{jk}\right]_n\E\left[\langle n\Delta Y,\Delta Y\rangle_{nh;k}\right]\|\Phi_{jk}\|_n^2+\frac{\eta_Y^2}{n}\left[\varphi_{jk},\varphi_{jk}\right]_n\E\left[\langle n\Delta X,\Delta X\rangle_{nh;k}\right]\|\Phi_{jk}\|_n^2\right.\\
&\left.\hspace*{1cm}+2\frac{\eta_{XY}}{n}\left[\varphi_{jk},\varphi_{jk}\right]_n\E\left[\langle n\Delta X,\Delta Y\rangle_{nh;k}\right]\|\Phi_{jk}\|_{n}^2\right)\\
&=\sum_{j=1}^{nh}w_{jk}^2\left(\|\Phi_{jk}\|_n^{-4}\frac{\eta_X^2\eta_Y^2+\eta_{XY}^2}{n^2}\left[\varphi_{jk},\varphi_{jk}\right]_n^2+(1+\rho_{kh}^2)(\sigma_{kh}^X\sigma_{kh}^Y)^2\right.\\
&\left.\hspace*{1cm}+n^{-1}\left(\eta_X^2(\sigma_{kh}^Y)^2+\eta_Y^2(\sigma_{kh}^X)^2+2\rho_{kh}\sigma_{kh}^Y\sigma_{kh}^X\eta_{XY}\right)\|\Phi_{jk}\|_n^{-2}\left[\varphi_{jk},\varphi_{jk}\right]_n\right)\\
&=\sum_{j=1}^{nh}w_{jk}^2\left(\|\Phi_{jk}\|_n^{-4}\,\frac{\eta_X^2\eta_Y^2+\eta_{XY}^2}{n^2}+(1+\rho_{kh}^2)(\sigma_{kh}^X\sigma_{kh}^Y)^2\right.\\
&\left.\hspace*{3cm}+\|\Phi_{jk}\|_n^{-2}\left(\frac{(\sigma_{kh}^Y\eta_{X})^2+(\sigma_{kh}^X\eta_{Y})^2+2\rho_{kh}\sigma_{kh}^X\sigma_{kh}^Y\eta_{XY}}{n}\right)\right).\end{align*}}
\noindent
We have used It\^{o} isometry, the features of model \eqref{E3} and Proposition \ref{propsbp}. To increase the readability we introduce the shortcut $I_{jk}$ and write the above term $\sum_j w_{jk}^2 I_{jk}^{-1}$, i.\,e.
\begin{align*}I_{jk}^{-1}&=\left(\|\Phi_{jk}\|_n^{-4}\,\frac{\eta_X^2\eta_Y^2+\eta_{XY}^2}{n^2}+(1+\rho_{kh}^2)(\sigma_{kh}^X\sigma_{kh}^Y)^2\right.\\
&\left.\hspace*{3cm}+\|\Phi_{jk}\|_n^{-2}\left(\frac{(\sigma_{kh}^Y\eta_{X})^2+(\sigma_{kh}^X\eta_{Y})^2+2\rho_{kh}\sigma_{kh}^X\sigma_{kh}^Y\eta_{XY}}{n}\right)\right)\,.\end{align*} 
Selecting appropriate weights with $\sum_{j=1}^{nh}w_{jk}=1$ on the blocks gives rise to an optimization problem with side condition. Minimizing the asymptotic variance 
yields oracle weights
\begin{align}
w_{jk}=\frac{I_{jk}}{\sum_{l=1}^{nh}I_{lk}}\,.\end{align}
Plugging in these weights the asymptotic variance on the $k$th block becomes $\sum_{j}I_{jk}^{-1}(I_{jk}^{2}/(\sum_l I_{lk})^2)=(\sum_lI_{lk})^{-1}$. Next, consider
\begin{align*}\frac{1}{\sqrt{n}h}\sum_{j=1}^{nh}I_{jk}=\frac{1}{\sqrt{n}h}\sum_{j=1}^{nh}\left(a+bn^2\sin^4{\left(\frac{j\pi}{2nh}\right)}+cn\sin^2{\left(\frac{j\pi}{2nh}\right)}\right)^{-1}\,,\end{align*}
with the shortcuts $a=(1+\rho_{kh}^2)(\sigma_{kh}^X\sigma_{kh}^Y)^2$, $b=16(\eta_X^2\eta_Y^2+\eta_{XY}^2)$ and $c=4\big((\eta_X\sigma_{kh}^Y)^2+(\eta_Y\sigma_{kh}^X)^2$\\ $+2\eta_{XY}\rho_{kh}\sigma_{kh}^X\sigma_{kh}^Y\big)$.
For $0<\alpha<3/8$, we obtain the bound
\begin{align*}\frac{1}{\sqrt{n}h}\sum_{j=n^{\nicefrac{5}{8}+\alpha}h}^{nh}I_{jk}&\le \frac{1}{\sqrt{n}h}nh\left(a+\frac{b}{2}n^2\frac{\pi^4n^{\nicefrac{20}{8}+4\alpha}h^4}{16n^4h^4}+\frac{c}{2}n\frac{\pi^2n^{\nicefrac{10}{8}+2\alpha}h^2}{4n^2h^2}\right)^{-1}\\
&=\sqrt{n}\left(a+\frac{b}{32}n^{\nicefrac{1}{2}+4\alpha}+\frac{c}{8}n^{\nicefrac{1}{4}+2\alpha}\right)^{-1}=\KLEINO(1)\,,\end{align*}
where we use that $\sin{x}\ge x/2$ on $(0,1)$, and further that by Taylor
\begin{align*}\frac{1}{\sqrt{n}h}\hspace*{-.1cm}\sum_{j=1}^{n^{\nicefrac{5}{8}+\alpha}h}\hspace*{-.1cm}I_{jk}&= \frac{1}{\sqrt{n}h}\hspace*{-.15cm}\sum_{j=1}^{n^{\nicefrac{5}{8}+\alpha}h}\hspace*{-.15cm}\left(\hspace*{-.1cm}a\hspace*{-.05cm}+\hspace*{-.05cm}bn^2\hspace*{-.05cm}\left(\frac{\pi^4j^4}{16n^4h^4}+\mathcal{O}\left(j^6n^{-6}h^{-6}\right)\right)\hspace*{-.1cm}+\hspace*{-.05cm}cn\hspace*{-.1cm}\left(\frac{j^2\pi^2}{4n^2h^2}\hspace*{-.05cm}+\hspace*{-.05cm}\mathcal{O}\hspace*{-.05cm}\left(j^4n^{-4}h^{-4}\right)\hspace*{-.1cm}\right)\hspace*{-.1cm}\right)^{-1}\\
&=\frac{1}{\sqrt{n}h}\sum_{j=1}^{n^{\nicefrac{5}{8}+\alpha}h}\left(a+(b/16) \pi^4(j/\sqrt{n}h)^4+(c/4) \pi^2(j/\sqrt{n}h)^2\right)^{-1}+\KLEINO(1)\,.\end{align*}
This means that uniformly for all $k$, the high frequencies $j\gsim n^{\nicefrac{5}{8}}$ do not contribute to the variance due to their decreasing weights and thus the sine functions may be approximated by the first order Taylor expansion.
The overall variance is with $h_0:=h\sqrt{n}\left(\eta_X^2\eta_Y^2+\eta_{XY}^2\right)^{-\nicefrac{1}{4}}$
\begin{align*}\operatorname{VAR}_n=\sum_{k=0}^{h^{-1}-1}h \frac{h_0}{\sqrt{n}}\left(\eta_X^2\eta_Y^2+\eta_{XY}^2\right)^{\nicefrac{1}{4}}\left(\sum_l I_{lk}\right)^{-1}\end{align*}
and hence, $\sqrt{n}\left(\eta_X^2\eta_Y^2+\eta_{XY}^2\right)^{-\nicefrac{1}{4}}\operatorname{VAR}_n$ and $n^{-1/2}h^{-1}\sum_l I_{lk}$ have the structure of Riemann sums. Because of $h_0\to\infty$ we can replace $j/h_0$ by an integration variable $z$ and we expect
\begin{equation}\label{EqRiemann}
n^{-1/2}h^{-1}\sum_{j=1}^{nh} I_{jk}\approx \int_0^{nh/h_0} \frac{1}{f_1(z)}dz
\end{equation}
with
\begin{equation}\label{Eqf1}
f_1(z)=f_1(\Sigma,{\mathbf H};z)=\pi^4z^4+\pi^2z^2\frac{\left(\eta_Y^2(\sigma^X)^2
+\eta_X^2(\sigma^Y)^2+2\eta_{XY}\rho\sigma^X\sigma^Y\right)}{\sqrt{\eta_X^2\eta_Y^2+\eta_{XY}^2}}+(1+\rho^2)(\sigma^X\sigma^Y)^2.
\end{equation}

Now for high-frequency asymptotics, letting $h\rightarrow 0$ for the piecewise constant approximation on the time-axis and also $h_0\rightarrow\infty$  as $n\rightarrow\infty$ for the asymptotics in the spectral frequency domain, we write the sums as integrals over positive functions in $[0,1]\times \mathds{R}_+$, which are constant on asymptotically vanishing rectangles, to a function $J(z,t)\:=\left(f_1\left(\Sigma_t,\mathbf{H};z\right)\right)^{-1}$.
This yields 
\begin{align}\label{eqSumIntConv}
n^{\nicefrac{1}{2}}\var\left(\widehat{IC}\right)=\int_0^1\left(\int_0^{nh/h_0} J\left(\lfloor z\rfloor_{h_0},\lfloor t\rfloor_h\right)\,dz\right)^{-1}dt\rightarrow\int_0^1\left(\int_0^{\infty}J(z,t)\,dz\right)^{-1}\,dt
\end{align}
by dominated convergence, since $J$ is continuous and  $J(z,t)\lsim (1+z^2)^{-2}$ uniformly in $t$.\\
The limit stated in Theorem \ref{theo3} for the oracle estimator is obtained by explicitly solving the integral, see Proposition \ref{integral} below. In the literature on nonparametric estimation methods for related and more general models, much mathematical effort is put in the proof of (stable) central limit theorems. For our Gaussian models the conclusion of asymptotic normality is direct. We can apply a standard i.\,i.\,d.\,triangular central limit theorem like Corollary 3.\,1 from \cite{hall}, verifying a Lyapunov condition with fourth moments.

To prove the adaptive version we observe first that an oracle estimator with  $\Sigma_{\floor{kh}_r}$ instead of $\Sigma_{kh}$ inserted into the optimal weights  attains still the optimal asymptotic variance whenever $r\to 0$. This follows from the same dominated convergence argument to establish \eqref{eqSumIntConv}.

Using the pilot estimator $\hat\Sigma_t$ with $\|\hat\Sigma-\Sigma\|_\infty=o_P(\delta_n)$, we shall show for the difference of estimators
\[ \sum_{k=0}^{h^{-1}-1}h\sum_{j=1}^{nh}(w_j(\hat\Sigma_{\floor{kh}_r})-w_j(\Sigma_{\floor{kh}_r}))
\|\Phi_{jk}\|_n^{-2}(\tilde x_{jk}\tilde y_{jk}-\eta_{XY}/n)=o_P(n^{-1/4})
\]
such that the adaptive result then follows from Slutsky's Lemma. We can work on the event $G_n=\{\|\hat\Sigma-\Sigma\|_\infty\le \delta_n\}$, since its complement has vanishing probability by the convergence rate for $\hat\Sigma$.\\
We can handle the arbitrary probabilistic dependence of $\hat\Sigma$ on $G_n$ by considering the maximal deterministic error and proving
\[ \sup_{\tilde\Sigma\ge 0:\|\tilde\Sigma-\Sigma\|_\infty\le\delta_n}
\sum_{k=0}^{h^{-1}-1}h\sum_{j=1}^{nh}(w_j(\tilde\Sigma_{\floor{kh}_r})-w_j(\Sigma_{\floor{kh}_r}))
\|\Phi_{jk}\|_n^{-2}(\tilde x_{jk}\tilde y_{jk}-\eta_{XY}/n)=o_P(n^{-1/4}).
\]
Splitting the terms along the coarse grid in $k$ and writing $Sym(1)=\{M\in\R^{2\times 2}\,|\, M \text{ symmetric},\|M\|\le 1\}$, it thus suffices to prove uniformly in $m=0,\ldots,r^{-1}-1$
\[ \sup_{M\in Sym(1)}\hspace*{-.1cm}\Big(r^{-1}hn^{1/4}\hspace*{-.1cm}
\sum_{k=mr/h}^{(m+1)r/h-1}\sum_{j=1}^{nh}(w_j(\Sigma_{mr}+\delta_nM)-w_j(\Sigma_{mr}))
\|\Phi_{jk}\|_n^{-2}(\tilde x_{jk}\tilde y_{jk}-\eta_{XY}/n)\Big)\hspace*{-.05cm}=\hspace*{-.05cm}o_P(1),
\]
which in turn follows by using $\sqrt{r}/\delta_n\to \infty$ and proving tightness in $C(Sym(1)\to \R,\|\,\cdot\,\|_{\infty})$ of the process
\[ \Big(r^{-1}hn^{1/4}
\sum_{j=1}^{nh}(w_j(\Sigma_{mr}+\sqrt{r}M)-w_j(\Sigma_{mr}))
\sum_{k=mr/h}^{(m+1)r/h-1}\|\Phi_{jk}\|_n^{-2}(\tilde x_{jk}\tilde y_{jk}-\eta_{XY}/n)\Big)_{M\in Sym(1)}.
\]
The standard argument is that uniform tightness in $M$ of processes $(X_n(M))$  implies that for matrices $M_n=r^{-1/2}(\hat\Sigma_{mr}-\Sigma_{mr})$ with $\|M_n\|\le r^{-1/2}\delta_n\to 0$, we obtain convergence
$X_n(M_n)-X_n(0)\stackrel{p}{\to} 0$.

By Kolmogorov's criterion (cf.\,Corollary 16.\,9 in \cite{kallenberg}), embedding $Sym(1)$ with its three degrees of freedom into $[-1,1]^3$, it suffices to find constants $\beta,C>0$ (uniformly in $m$) such that for all $M,M'\in Sym(1)$:
\begin{align}\notag &\E\Big[\Big(r^{-1}hn^{1/4}
\sum_{k=mr/h}^{(m+1)r/h-1}\sum_{j=1}^{nh}(w_j(\Sigma_{mr}+\sqrt{r}M)-w_j(\Sigma_{mr}+\sqrt{r}M'))
\|\Phi_{jk}\|_n^{-2}(\tilde x_{jk}\tilde y_{jk}-\eta_{XY}/n)\Big)^4\Big]\\
&\label{ada}\le C\|M-M'\|^{3+\beta}.
\end{align}
For this we bound the gradient $\nabla w_j(\Sigma)$ uniformly in $j,n,h$ and $\Sigma$. Writing $w_j(\Sigma)=C(\Sigma)W_j(\Sigma)$ from formula \eqref{woracle} with $W_j(\Sigma)=I_{jk}=(\var(\tilde x_{jk}\tilde y_{jk}))^{-1}$, emphasizing independence of the block $k$, and the normalization factor $C(\Sigma)$, we derive
\begin{align*}
|\nabla W_j(\Sigma)|&\lesssim W_j(\Sigma)^2(1+j^2/h_0^2)\lesssim W_j(\Sigma),\\
|\nabla C(\Sigma)|&= \left|\nabla \left(\sum_jW_j(\Sigma)\right)^{-1}\right|\lesssim C(\Sigma)^2\sum_j|\nabla W_j(\Sigma)|\lesssim C(\Sigma).
\end{align*}
This yields the uniform estimate $|\nabla w_j(\Sigma)|\lesssim h_0^{-1}(1+j^4/h_0^4)^{-1}$.

By independence and Gaussianity of the factors in $(\tilde x_{jk}\tilde y_{jk}-\eta_{XY}/n)$, the left-hand side in \eqref{ada} above is of order
\begin{align*}
&\left(\hspace*{-.15cm}\var\Big(r^{-1}hn^{1/4}
\hspace*{-.1cm}\sum_{k=mr/h}^{(m+1)r/h-1}\hspace*{-.1cm}\sum_{j=1}^{nh}(w_j(\Sigma_{mr}+\sqrt{r}M)\hspace*{-.05cm}-\hspace*{-.05cm}w_j(\Sigma_{mr}+\sqrt{r}M'))
\|\Phi_{jk}\|_n^{-2}(\tilde x_{jk}\tilde y_{jk}\hspace*{-.05cm}-\hspace*{-.05cm}\eta_{XY}/n)\Big)\hspace*{-.15cm}\right)^2\\
& \lesssim \Big(r^{-2} h^2 n^{1/2}\sum_{j=1}^{nh}\sum_{k=mr/h}^{(m+1)r/h-1}
\|\nabla w_j\|_\infty^2\|\sqrt{r}(M-M')\|^2
(j^4/h^4)\var(\tilde x_{jk}\tilde y_{jk})\Big)^2\\
&\lesssim
\Big(r^{-2} h^2 n^{1/2}\sum_{j=1}^{nh}(r/h)
(h_0(1+j^4/h_0^4))^{-2}r\|M-M'\|^2
(1+j^2/h_0^2)^2\Big)^2\\
&\lesssim
\|(M-M')\|^4\,.
\end{align*}


The following Proposition completes the proof of Theorem \ref{theo3}. For the computation of the Riemann integrals we use some concepts from complex analysis.
\begin{prop}\label{integral}
Consider the functions $f_1:\C\rightarrow\C$ generalizing the function \eqref{Eqf1} and $f_2:\C\rightarrow\C$ with
\[
f_2(z):=f_2(\rho,\sigma;z)=\pi^4z^4+2\sigma^2\pi^2z^2+(1+\rho^2)\sigma^4=f_1\left(\begin{pmatrix}\sigma&\rho\\\rho &\sigma\end{pmatrix},
\begin{pmatrix} \eta & 0\\ 0 &\eta\end{pmatrix};z\right)\,,
\]
which depend on parameters $\rho$ and positive $\sigma^{\,\cdot\,},\eta_{\,\cdot\,}$. For the improper integrals along the positive real line in the case $\sigma^{\,\cdot\,}>0,\rho\ne 0,$ the following identities hold true:
\begin{subequations}
\begin{align}\label{integral1}\int_0^{\infty}\frac{1}{f_2(x)}\,dx=\frac{1}{2\sigma^3\rho(1+\rho^2)^{\nicefrac{1}{4}}}\,\sin\left(\frac{1}{2}\left(\operatorname{Arg}\left(\ii -\rho\right)-\frac{\pi}{2}\right)\right)\,,\end{align}
\begin{align}\label{integral2}\int_0^{\infty}\frac{1}{f_1(x)}\,dx=\frac{1}{\sqrt{2}\sqrt{A^2-B}\sqrt{B}}\left(\sqrt{A+\sqrt{A^2-B}}-\operatorname{sgn}(A^2-B)\sqrt{A-\sqrt{A^2-B}}\right)\,,\end{align}
\end{subequations}
where $A$ and $B$ are short expressions for the terms
\begin{align*}
A&=\left(\frac{\eta_Y^2(\sigma^X)^2+\eta_X^2(\sigma^Y)^2+2\eta_{XY}\rho\sigma^X\sigma^Y}{\sqrt{\eta_X^2\eta_Y^2+\eta_{XY}^2}}\right),
\quad B=4(\sigma^X\sigma^Y)^2(1+\rho^2)\,,
\end{align*}
$\operatorname{Arg}(z)$, with $\operatorname{Arg}(z)=\arctan(\operatorname{Im}(z)/\operatorname{Re}(z))$ for $\operatorname{Re}(z)>0$\,, denotes the argument of a complex number and we determine to take the root located in the upper half plane $\H=\{z\in\C|\operatorname{Im}(z)> 0\}$ in \eqref{integral2} in the case that $A^2-B< 0$.\\
For $\rho=0$ and strictly positive $\sigma^{\,\cdot\,}$ (and $\eta_{\,\cdot\,}$), the integrals yield
\begin{subequations}
\begin{align}\label{integral3}\int_0^{\infty}\frac{1}{f_2(x)}\,dx=\frac{1}{4\sigma^3}\,,\end{align}
\begin{align}\label{integral4}\int_0^{\infty}\frac{1}{f_1(x)}\,dx=\frac{1}{2\eta_Y(\eta_X^2\eta_Y^2+\eta_{XY}^2)^{-1/4}(\sigma^X)^2\sigma^Y+2\eta_X(\eta_X^2\eta_Y^2+\eta_{XY}^2)^{-1/4}(\sigma^Y)^2\sigma^X}\,.\end{align}
\end{subequations}
\end{prop}
\begin{proof}
The meromorphic functions $f_1^{-1}:\C\rightarrow\C,\,f_2^{-1}:\C\rightarrow\C$ have four simple poles in the complex plane, since $f_1,\,f_2$ each has four simple non-real zeros. We can apply a specific version of the residue theorem (cf.\,Theorem 7.10\,in Chapter III of \cite{busam}) to evaluate the above improper real integrals. We restrict ourselves to the case $\rho\ne 0$ for which the solutions of \eqref{integral1} and \eqref{integral2} are not feasible using algebra programs or standard integral tables.\\
We first give the proof of \eqref{integral1} for the simplified function $f_2^{-1}$. The zeros of $f_2$ are
\begin{align*}z_{2;1,2,3,4}=\pm \frac{\exp{\left(\ii \,\pi/4\right)}}{\pi}\sqrt{\ii \pm\rho}\,\sigma\end{align*}
and are located symmetrically on a disk around the null in the complex plane. The residue theorem allows to calculate the integral $\int f_2^{-1}$ along the real line by the limit of a  curve integral over a half-disk in the upper half plane. Since $f_2$ is even on the real line and $z_{2;1}$ and $z_{2;4}$ are the poles in the upper half plane, we obtain:
\begin{align*} \int_0^{\infty}\frac{1}{f_2(x)}\,dx&=\pi\ii\left(\operatorname{Res}\left(f_2^{-1};z_{2;1}\right)+\operatorname{Res}\left(f_2^{-1};z_{2;4}\right)\right)\\
&=\pi\ii\left(\left(4\exp{(\ii \pi/4)}\pi\sqrt{\rho+\ii}\,\sigma^3+4\exp{(\ii \,3\pi/4)}\pi(\rho+\ii)^{\nicefrac{3}{2}}\sigma^3\right)^{-1}\right. \\
&~~~~~~\left. -\left(4\exp{(\ii \pi/4)}\pi\sqrt{\ii-\rho}\,\sigma^3+4\exp{(\ii \,3\pi/4)}\pi(\ii -\rho)^{\nicefrac{3}{2}}\sigma^3\right)^{-1}\right)\\
&=\frac{1}{4\sigma^3\rho}(-1)^{\nicefrac{3}{4}}\left((\ii-\rho)^{-\nicefrac{1}{2}}-(\ii+\rho)^{-\nicefrac{1}{2}}\right)\\
&=\left(2\sigma^3\rho(\rho^2+1)^{\nicefrac{1}{4}}\right)^{-1}\,\sin{\left(\frac{1}{2}\left(\operatorname{Arg}(\ii-\rho)-\frac{\pi}{2}\right)\right)}\,.\end{align*}
In this proof we always use the unique square root in the upper half plane of complex numbers (and the usual definition for real numbers).\\
The analysis for the general case $f_1$ is a bit more involved, since depending on the parameters $\rho,\sigma^{\,\cdot\,}$ and the ratios $\eta_X^2/\sqrt{\eta_X^2\eta_Y^2+\eta_{XY}^2}$, $\eta_Y^2/\sqrt{\eta_X^2\eta_Y^2+\eta_{XY}^2}$
for the zeros of $f_1$:
\begin{align*}z_{1;1,2,3,4}=\pm\frac{1}{\sqrt{2}\,\pi}\sqrt{-A\pm\sqrt{A^2-B}}\,,\end{align*}
\begin{figure}[t]
\fbox{\includegraphics[width=6.7cm]{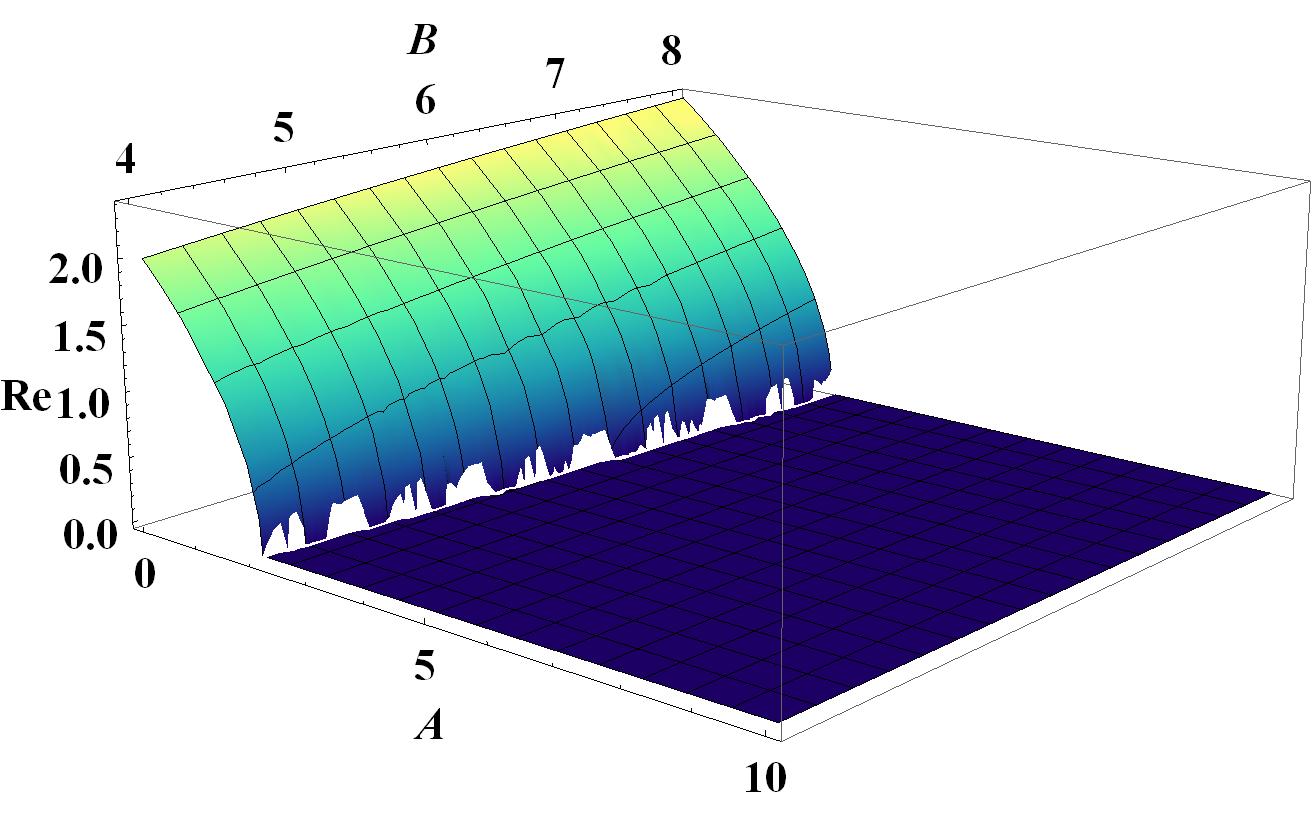}}~~~~\fbox{\includegraphics[width=7.3cm]{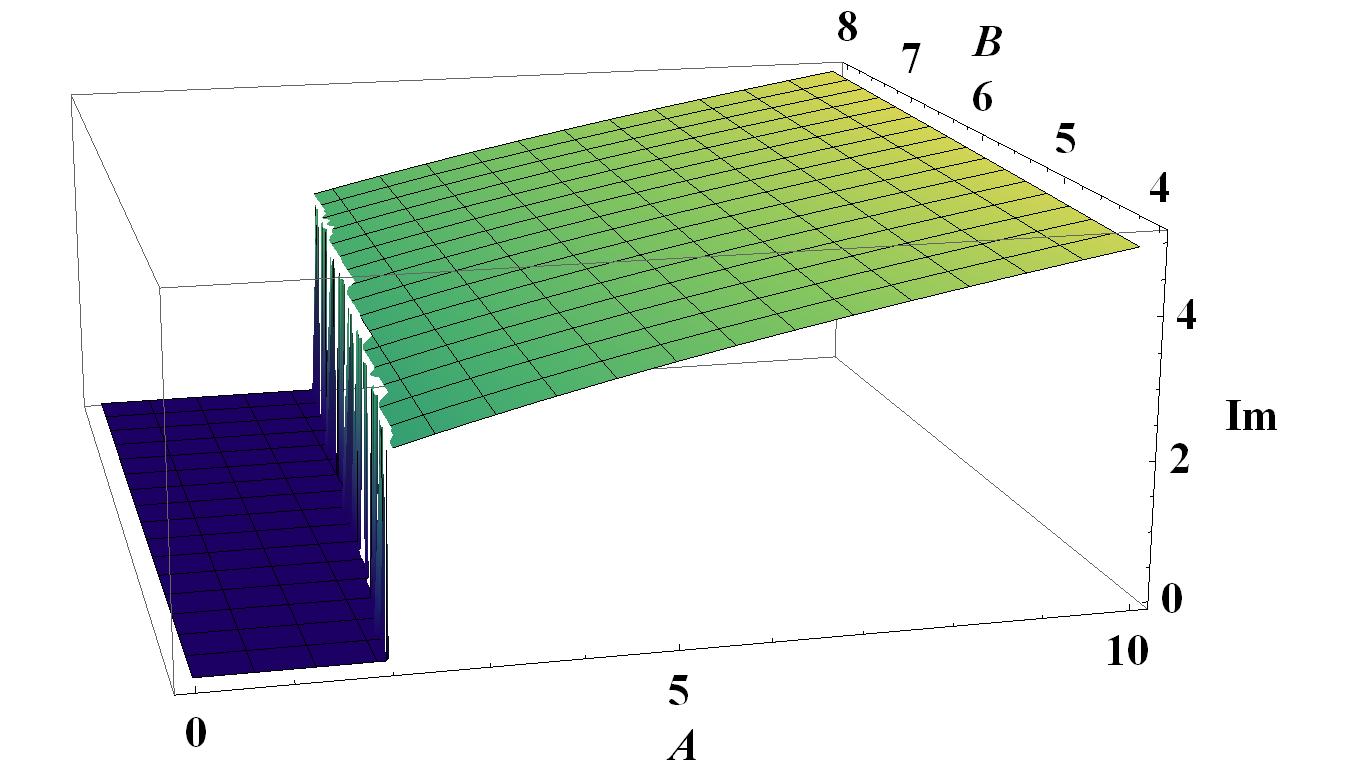}}\noindent
\caption{\label{fig:complex3d}Real and imaginary part of $\sqrt{-A-\ii\sqrt{B-A^2}}+\sqrt{-A+\ii\sqrt{B-A^2}}$.}\end{figure}
\hspace*{-.1cm}it holds true that either $z_{1;1}(\text{``++''})$ and $z_{1;4}(\text{``- -''})$ or $z_{1;1}(\text{``++''})$ and $z_{1;2}(\text{``+ -''})$ are located in the upper half plane. This role change dependent on whether $A^2-B$ is positive or negative is illustrated in Figure \ref{fig:complex3d}, in which the interesting factor appearing in the solution of the integral is depicted for a possible range of values for $A$ and $B$ in a certain codomain of $\rho,\sigma^X,\sigma^Y,\eta_X/\eta_Y$ for $\eta_{XY}=0$. Using the above convention for square roots, the left-hand side of \eqref{integral2} yields
\begin{align*}\int_0^{\infty}f_1^{-1}(x)\,dx&=\frac{\ii}{\sqrt{2}}\frac{1}{\sqrt{A^2-B}}\left((-A+\sqrt{A^2-B})^{-\nicefrac{1}{2}}\pm(-A-\sqrt{A^2-B})^{-\nicefrac{1}{2}}\right)\\
&=\frac{1}{\sqrt{2}}\frac{1}{\sqrt{B-A^2}}\frac{1}{\sqrt{B}}\left(\sqrt{-A-\ii\sqrt{B-A^2}}+\operatorname{sgn}(B-A^2)\sqrt{-A+\ii\sqrt{B-A^2}}\right)\\
&=\frac{1}{\sqrt{2}} \frac{1}{\sqrt{A^2-B}}\frac{1}{\sqrt{B}}\left(\sqrt{A+\sqrt{A^2-B}}-\operatorname{sgn}(A^2-B)\sqrt{A-\sqrt{A^2-B}}\right)\,.\end{align*}
In the first line ``$\pm$'' indicates that depending on the parameters there are two different solutions. As visualized in Figure \ref{fig:complex3d}, the right factor in the second line is purely real if $B-A^2>0$ and purely imaginary if $B-A^2<0$. The expressions in the second and third line hence give the positive real solution in each case. In the case $B-A^2>0$, we can write the solution $\sqrt{2}(B-A^2)^{-\nicefrac{1}{2}}B^{-\nicefrac{1}{4}}\cos{\left(\frac{1}{2}\operatorname{Arg}(A+\ii\sqrt{B-A^2})\right)}$ similiarly to \eqref{integral1} above.
\end{proof}
\bibliographystyle{mychicago}
\bibliography{literaturabr}\vspace*{1cm}
\noindent
Markus Bibinger, Institute of Mathematics, Humboldt-Universität zu Berlin, Unter den Linden 6, 10099 Berlin, Germany, E-mail: bibinger@math.hu-berlin.de\\
Markus Rei\ss, Institute of Mathematics, Humboldt-Universität zu Berlin, Unter den Linden 6, 10099 Berlin, Germany, E-mail: mreiss@math.hu-berlin.de

\end{document}